\documentclass[a4paper,11pt]{amsart}
\usepackage{amsmath,amsfonts,amsthm,amssymb,enumerate,ifpdf}
\usepackage[pdftex]{graphicx}

\numberwithin{equation}{section}
\numberwithin{figure}{section}

\newtheorem{theorem}{Theorem}[section]
\newtheorem{corollary}[theorem]{Corollary}
\newtheorem{proposition}[theorem]{Proposition}

\newtheorem{lemma}[theorem]{Lemma}

{
\theoremstyle{definition}
\newtheorem{definition}[theorem]{Definition}
\newtheorem{example}[theorem]{Example}
\newtheorem{remark}[theorem]{Remark}
}

\title{Transformations of partial matchings  
}
\author{Inasa Nakamura}
\address{Graduate School of Mathematical Sciences, The University of Tokyo\newline
3-8-1 Komaba, Tokyo 153-8914, Japan}
\email{inasa@ms.u-tokyo.ac.jp}

\address{Current Address: Faculty of Electrical, Information and Communication Engineering, 
Institute of Science and Engineering, 
Kanazawa University \newline
Kakumamachi, Kanazawa, 920-1192, Japan}
\email{inasa@se.kanazawa-u.ac.jp}


\begin{document}  
\begin{abstract}
We consider partial matchings, which are finite graphs consisting of edges and vertices of  degree zero or one. We consider transformations between two states of partial matchings. We introduce a method of presenting a transformation between partial matchings. We introduce the notion of the lattice presentation of a partial matching, and the lattice polytope associated with a pair of lattice presentations, and we investigate transformations with minimal area. 
\end{abstract}
\maketitle

\section{Introduction}\label{sec1}
A {\it partial matching}, or a {\it chord diagram}, 
is a finite graph consisting of edges and vertices of degree zero or one, with the vertex set $\{1,2,\ldots, m\}$ for a positive integer $m$ \cite{Reidys}. A partial matching is used to present secondary structures of polymeric molecules such as RNAs (see Section \ref{sec2-1}). The aim of this paper is to establish mathematical basics for discussing partial matchings, from a viewpoint of lattice presentations. 

In this paper, we consider transformations between two states of partial matchings. 
In Section \ref{sec2}, we introduce the lattice presentation of a partial matching, and we clarify correspondence between structures of chord diagrams and lattice presentations. In Section \ref{sec3}, we introduce the notion of the lattice polytope associated with a pair of partial matchings, and we discuss the equivalence of lattice polytopes. In Section \ref{sec4}, we introduce transformations of lattice presentations and lattice polytopes, and the area of a transformation. We give a lower estimate of the area of a transformation by using the area of a lattice polytope, and in certain cases we construct  transformations with minimal area (Theorem \ref{thm3-10}). In Section \ref{sec-red}, we introduce the notion of the reduced graph of a lattice polytope, and we show that there exists a transformation of a lattice polytope with minimal area if and only if its reduced graph is an empty graph (Theorem \ref{thm-red}). Section \ref{sec4-3} is devoted to showing lemmas and propositions. In Section \ref{sec5}, we consider simple connected lattice polytopes. We show that when a lattice polytope $P$ is connected and simple, a certain division of $P$ into $n$ rectangles presents a transformation with minimal area, and in certain cases, any transformation with minimal area is presented by such a division of $P$ (Corollary \ref{cor5-2}). 

\section{Partial matchings and our motivation}\label{sec2-1}

Partial matchings present secondary structures of polymeric molecules such as RNAs. An RNA is a single-strand chain of simple units of nucleotides called nucleobases, with a backbone with an orientation from 5'-end to 3'-end, which folds back on itself. The nucleobases consist of 4 types, 
guanine (G), uracil (U), adenine (A), and cytosine (C), and there are interactions between nucleobases: adenine and uracil, guanine and cytosine, which form A-U, G-C  
 base pairs. The secondary structure of an RNA is the information of base pairs. For an RNA strand, regard each nucleobase as a vertex, and label the vertices by integers $1,2,\ldots$ from 5'-end to 3'-end, and connect two vertices forming each base pair with an edge. Then we have a chord diagram presenting the secondary structure. 
Chord diagrams have been studied to investigate RNA secondary structures, which consist of nesting structures formed by \lq\lq parallel'' bonds, and pseudo-knot structures containing \lq\lq cross-serial'' bonds. In particular, a special kind of structure called a \lq\lq $k$-noncrossing structure'' plays an important role and enumerations of $k$-noncrossing structures are investigated \cite{CDDSY, JQR, JR, Reidys}. 
Our motivation of this research was to give a new method of investigating RNA secondary structures. 

Partial matchings are used to predict the most possible forms of RNA secondary structures with optimal free energy, by means of dynamic programming, calculating energy for every possible state of partial matchings, and investigating one partial matching at a time; there are many references, for example see \cite{PM, Reidys, Tinoco-al, Zuker-Sankoff}. 
In this paper, from a different viewpoint, we 
focus on investigating paths between two fixed states of partial matchings. 

\section{Lattice presentations}\label{sec2}

In this section, we introduce the notion of the lattice presentation of a partial matching, and we see the correspondence between structures of chord diagrams and lattice presentations.

We recall the precise definition of a partial matching, or a chord diagram. 
A {\it partial matching}, or a {\it chord diagram},
is a finite graph consisting of edges and vertices of degree zero or one, with the vertex set $\{1,2,\ldots, m\}$ for a positive integer $m$. 
A chord diagram is represented by drawing the vertices $1,2,\ldots, m$ in the horizontal line and the edges in the upper half plane. We denote by $(x,y)$ ($x, y \in \{1,2,\ldots,m\}$, $x < y$) the edge connecting the vertices $x$ and $y$ and call it an {\it arc}, and we call a vertex with degree zero an {\it isolated vertex}. In this paper, we use the term \lq\lq partial matching'' when we consider its lattice presentation, and the term \lq\lq chord diagram'' when we consider the chord diagram itself.

For the $xy$-plane $\mathbb{R}^2$, 
put $C=\{(x,y)\in \mathbb{R}^2 \mid  x, y>0\}$, and 
we denote by $C_+$ (respectively $C_-$) the half plane $\{(x,y)\in C \mid x\leq y\}$ (respectively $\{(x,y)\in C \mid x\geq y\}$). We call a point $(x,y)$ of $C$ such that $x$ and $y$ are integers a {\it lattice point}. 

\begin{definition}
For a partial matching $\delta$, 
the {\it lattice presentation} $\Delta$ of $\delta$ is a set of lattice points in $C \backslash \partial C_+$ such that each element $(x,y)$ of $\Delta$ presents an arc $(x,y)$ when $x<y$ (respectively an arc $(y,x)$ when $y<x$) of $\delta$. Note that for each arc $(x,y)$ of $\delta$, we have two points $(x,y) \in C_+$ and $(y,x) \in C_-$ of $\Delta$. 
\end{definition}

\begin{example}
The left figure of Figure \ref{g-1} is a chord diagram $\delta$ with five arcs $(1, 5)$, $(2,4)$, $(3,7)$, $(6,8)$, $(10, 12)$, and two isolated vertices $9$ and $11$, and the right figure of Figure \ref{g-1} is the lattice presentation of $\delta$. 
\end{example}

\begin{figure}[ht]
\centering
\includegraphics*[height=4cm]{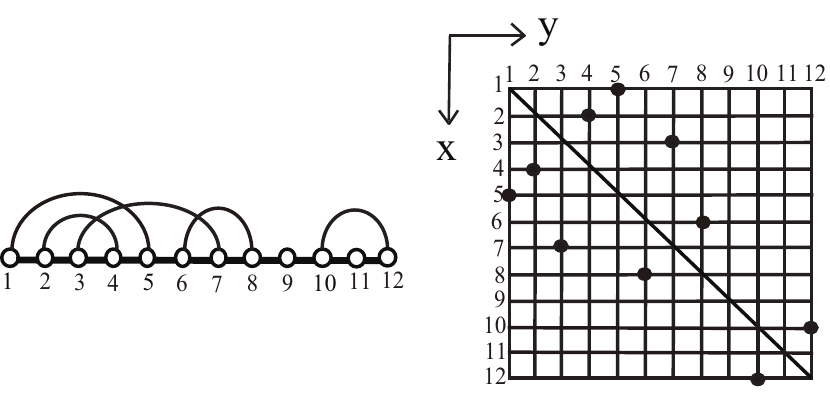}
\caption{A chord diagram (left) and the lattice presentation (right).}
\label{g-1}
\end{figure}

Partial matchings are investigated by considering special structures called $k$-nesting and $k$-noncrossing; for example, see \cite{CDDSY, JQR, JR, Reidys}. We see the correspondence of a partial matching with such a structure and the lattice presentation. 
By definition of lattice presentation of a partial matching, we have the following propositions. 
For two points $v_1, v_2$ of $C$, we denote by $R(v_1, v_2)$ the rectangle in $C$ whose diagonal vertices are $v_1$ and $v_2$.  
For a point $v=(x,y) \in C_+$, put $v^*=(y,x) \in C_-$. 
  
Two arcs $(x_1, y_1)$ and $(x_2, y_2)$ of a chord diagram are said to be {\it separated} if the intervals $[x_1, y_1]$ and $[x_2, y_2]$ in $\mathbb{R}$ are disjoint. Two arcs are said to be {\it non-separated} if they are not separated. 

\begin{proposition}\label{prop2-3}

For a lattice presentation $\Delta=\{ v_1, v_2, v_1^*,  v_2^* \}$ with $v_j \in C_+, v_j^* \in C_-$ $(j=1,2)$, the following conditions are mutually equivalent (see Figure \ref{fig-2}).

\begin{enumerate}[$(1)$]
\item
A lattice presentation $\Delta$ presents separated arcs.
\item
We have $R(v_1, v_2) \not\subset C_+$.
\item
We have $R(v_1, v_2) \cap R(v_1^*,v_2^*) \neq \emptyset$.
\end{enumerate}
\end{proposition}

\begin{proof}
Put $v_1=(x_1, y_1)$ and $v_2=(x_2, y_2)$. Let us assume $x_1<x_2$. Then, 
two arcs $(x_1, y_1)$ and $(x_2, y_2)$ are separated if and only if $y_1<x_2$. Let us denote the vertices of the rectangle $R(v_1, v_2)$ by $v_1=(x_1, y_1), u_1=(x_1, y_2), u_2=(x_2, y_1)$, and $v_2=(x_2, y_2)$. Since $x_j<y_j$ ($j=1,2$), with the assumption $x_1<x_2$, we see that $v_1, u_1, v_2 \in C_+$, and $y_1<x_2$ if and only if $u_2=(x_2, y_1) \in C_-$, which is equivalent to the condition that $R(v_1, v_2) \not\subset C_+$. Since $R(v_1^*, v_2^*)$ is the mirror reflection of $R(v_1, v_2)$ with respect to the line $\partial C_+$, and $v_1$ and $v_2$ are not in $\partial C_+$,  $R(v_1, v_2) \not\subset C_+$ if and only if $R(v_1, v_2) \cap R(v_1^*,v_2^*) \neq \emptyset$. 
\end{proof}

\begin{figure}[ht]
\centering
\includegraphics*[height=4cm]{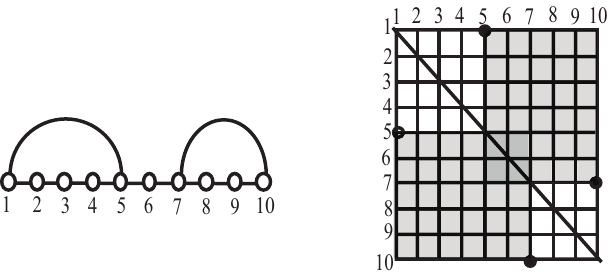}
\caption{Separated arcs of a chord diagram and the lattice presentation, where we shadow rectangles $R(v_1, v_2)$ and $R(v_1^*, v_2^*)$. For this figure, $v_1=(1,5)$ and $v_2=(7,10)$.}
\label{fig-2}
\end{figure}

For a rectangle $R(v_1, v_2)$, we say it is {\it of type I} (respectively {\it of type II, III, IV}) if the vector from $v_1$ to $v_2$, $v_2-v_1\in \mathbb{R}^2$ is in the first (respectively second, third, fourth) quadrant. 

Let $k$ be a positive integer. 
A chord diagram is called {\it $k$-nesting} if its arcs consist of $k$ distinct arcs $(x_1, y_1)$, $(x_2, y_2)$, $\ldots$, $(x_k, y_k)$ such that 
\begin{equation}\label{nest}
x_1<x_2<\cdots<x_k<y_k<y_{k-1}<\cdots<y_1, 
\end{equation}
see Figure \ref{fig-3}.

\begin{proposition}
For a lattice presentation $\Delta=\{ v_1, \ldots, v_k, v_1^*, \dots, v_k^* \}$ with $v_j \in C_+, v_j^* \in C_-$ $(j=1,\dots, k)$, the following conditions are mutually equivalent.

\begin{enumerate}[$(1)$]
\item
A lattice presentation $\Delta$ presents a $k$-nesting chord diagram.

\item
Each rectangle $R(v_{i}, v_{j})$ is of type II or IV for $i,  j=1,\dots, k, i \neq j$. 

\item
Changing the indices if necessary, we have the following: $R(v_{j}, v_{j+1})$ is of type IV for each $j=1,\dots, k-1$.
\end{enumerate}
\end{proposition}

\begin{proof}
Put $v_j=(x_j, y_j)$ ($j=1,\dots, k$). 
The relation (\ref{nest}) is equivalent to $x_{i}<x_{j}<y_{j}<y_{i}$ when $i<j$ (respectively $x_{j}<x_{i}<y_{i}<y_{j}$ when $i>j$). Since $v_i, v_j \in C_+$, this is equivalent to $x_{i}<x_{j}$ and $y_{j}<y_{i}$ when $i<j$ (respectively $x_{j}<x_{i}$ and $y_{i}<y_{j}$ when $i>j$), which is equivalent to the condition that $R(v_{i}, v_{j})$ is of type IV when $i<j$ (respectively of type II when $i>j$). Thus the conditions 1 and 2 are equivalent. 
Again, the relation (\ref{nest}) is equivalent to $x_{j}<x_{j+1}<y_{j+1}<y_{j}$ for $j=1,\dots, k-1$, which is equivalent to the condition that $R(v_{j}, v_{j+1})$ is of type IV for each $j=1,\dots, k-1$. 
Thus the conditions 1 and 3 are equivalent. 
\end{proof}
\begin{figure}[ht]
\centering
\includegraphics*[height=4cm]{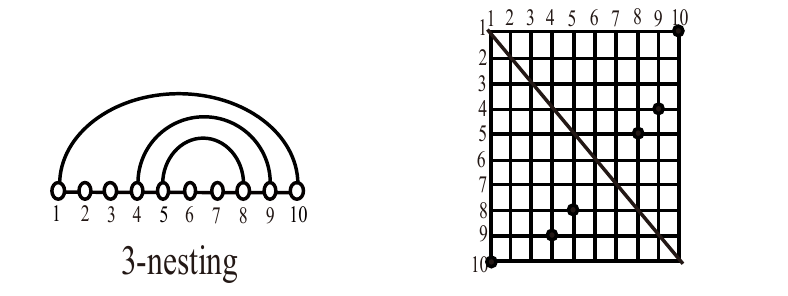}
\caption{A  nesting  chord diagram and the lattice presentation.}
\label{fig-3}
\end{figure}
 
A chord diagram is called {\it $k$-crossing} if its arcs consist of $k$ distinct arcs $(x_1, y_1), (x_2, y_2), \ldots, (x_k, y_k)$ such that 
\begin{equation}\label{crossing}
x_1<x_2<\cdots<x_k<y_1<y_2<\cdots<y_k, 
\end{equation}
see Figure \ref{fig-4}. 
A chord diagram is called {\it $k$-noncrossing} if there exists no $k$-crossing subgraph. 

\begin{proposition}\label{prop2-5}
For a lattice presentation $\Delta=\{ v_1, \ldots, v_k, v_1^*, \dots, v_k^* \}$ with $v_j \in C_+, v_j^* \in C_-$ $(j=1,\dots, k)$, the following conditions are mutually equivalent.

\begin{enumerate}[$(1)$]
\item
A lattice presentation $\Delta$ presents a $k$-crossing chord diagram. 

\item
Each rectangle $R(v_{i}, v_{j})$ is of type I or III and contained in $C_+$ for $i,  j=1,\dots, k, i \neq j$. 

\item
Changing the indices if necessary, we have the following: $R (v_1, v_k)$ is contained in $C_+$ and $R(v_{j}, v_{j+1})$ is of type I for each $j=1,\dots, k-1$.
\end{enumerate}
\end{proposition}

\begin{proof}
Put $v_j=(x_j, y_j)$ ($j=1,\dots, k$). 
The relation (\ref{crossing}) is equivalent to $x_{i}<x_{j}<y_{i}<y_{j}$ when $i<j$ (respectively $x_{j}<x_{i}<y_{j}<y_{i}$ when $i>j$). 
This is equivalent to $x_{i}<x_{j}$, $y_{i}<y_{j}$, and $x_{j}<y_{i}$ when $i<j$ (respectively $x_{j}<x_{i}$, $y_{j}<y_{i}$, and $x_i<y_j$ when $i>j$). When $i<j$, $x_{i}<x_{j}$ and $y_{i}<y_{j}$ if and only if $R(v_{i}, v_{j})$ is of type I, and since $R(v_i, v_j)$ is a rectangle whose vertices are $(x_i, y_i), (x_i, y_j), (x_j, y_i)$ and $(x_j, y_j)$, with the condition $x_i<x_j$ and $y_i<y_j$, we see that $x_j<y_i$ if and only if $R(v_i, v_j)$ is contained in $C_+$. Thus, by the same argument, we see that the condition 1 is equivalent to the condition that $R(v_{i}, v_{j})$ is of type I when $i<j$ (respectively of type III when $i>j$) and contained in $C_+$. Thus the conditions 1 and 2 are equivalent. 
Again, the relation (\ref{crossing}) is equivalent to $x_{j}<x_{j+1}$ and $y_{j}<y_{j+1}$ for $j=1,\dots, k-1$, and $x_k<y_1$.  Then, by the same argument, we see that the conditions 1 and 3 are equivalent. 
\end{proof}

\begin{figure}[ht]\centering
\includegraphics*[height=4cm]{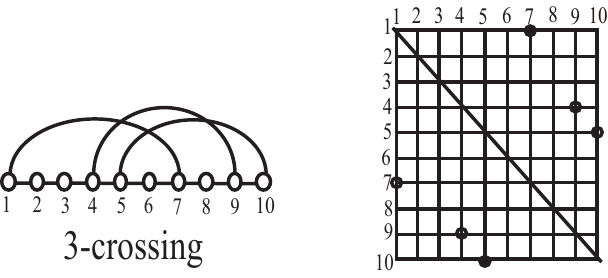}
\caption{A crossing chord diagram and the lattice presentation.}
\label{fig-4}
\end{figure}

\section{Lattice polytopes}\label{sec3}
In this section, we give the definition of the lattice polytope associated with two partial matchings. 
We take the standard basis $\mathbf{e}_1, \mathbf{e}_2$ of the $xy$-plane $\mathbb{R}^2$, where $\mathbf{e}_1=(1,0)$ and $\mathbf{e}_2=(0,1)$. 
For two distinct points $v, w$ of $\mathbb{R}^2$, we denote by $\overline{vw}$ the segment connecting $v$ and $w$. We say a segment $\overline{vw}$ is {\it in the $x$-direction} or {\it in the $y$-direction} if $w=v+x\mathbf{e}_1$ or $w=v+y\mathbf{e}_2$ respectively for some $x,y \in \mathbb{R}$: we say a segment is in the $x$-direction or in the $y$-direction if it is parallel to the $x$-axis or the $y$-axis, respectively. We denote by $R(v, w)$ the rectangle in $\mathbb{R}^2$ whose diagonal vertices are $v$ and $w$. 
For a point $v=(x_1,y_1)$ of $\mathbb{R}^2$, we call $x_1$ (respectively $y_1$) the {\it $x$-component} (respectively the {\it $y$-component}) of $v$. 
For a set of points of $\mathbb{R}^2$, $\{ v_1, \ldots, v_n\}$  with $v_j=(x_j, y_j)$ ($j=1,\dots, n$), we call the set $\{ x_1, \ldots, x_n, y_1,\ldots, y_n\}$, the {\it set of $x$, $y$ components} of $\{ v_1,\ldots, v_n\}$. 
The set of $x$, $y$-components of a lattice presentation $\Delta$ with $n$ arcs is a set $X$ of $2n$ distinct points of $\mathbb{R}$, which, as a multiset, consists of two copies of $X$. 

\begin{definition}
A {\it lattice polytope} is a polytope $P$ consisting of a finite number of vertices of degree zero (called {\it isolated vertices}) with multiplicity $2$, vertices of degree $2$ with multiplicity $1$, edges, and faces satisfying the following conditions. 

\begin{enumerate}[$(1)$]
\item
Each edge is either in the $x$-direction or in the $y$-direction. 

\item
The $x$-components (respectively $y$-components) of isolated vertices and edges in the $x$-direction (respectively $y$-direction)  are distinct. 

\item
The boundary $\partial P$ is equipped with a coherent orientation as an immersion of a union of several circles, where we call the union of edges of $P$ the {\it boundary} of $P$, and denote it by $\partial P$. 
\end{enumerate}
 
The set of vertices are divided to two sets $X_0$ and $X_1$ 
such that each edge in the $x$-direction is oriented from a vertex of $X_0$ to a vertex of $X_1$, and isolated vertices are both in $X_0$ and $X_1$ with multiplicity 1. We denote $X_0$ (respectively $X_1$) by $\mathrm{Ver}_0(P)$ (respectively $\mathrm{Ver}_1(P)$), and call it the set of {\it initial vertices} (respectively {\it terminal vertices}). 
\end{definition}

In graph theory, a lattice polytope is defined as a polytope whose vertices are lattice points \cite{Barvinok2, Diestel}, but in this paper, when we say that $P$ is a lattice polytope, we assume that each edge of $P$ is either in the $x$-direction or in the $y$-direction. 

For a lattice polytope $P$, we denote by $-P^*$ the orientation-reversed mirror reflection of $P$ with respect to the line $x=y$.

\begin{proposition}\label{prop1}
Let $\Delta, \Delta'$ be two lattice presentations with the same set of $x$, $y$-components. 
Then, $\Delta$ and $\Delta'$ form a pair of lattice polytopes $P, -P^*$ satisfying that each edge is either in the $x$-direction or in the $y$-direction such that it  connects a point of $\Delta$ and a point of $\Delta'$. See Figure \ref{fig3-1}. 
\end{proposition}

\begin{definition}\label{def3}
We call a lattice polytope $P$ as in Proposition \ref{prop1} the {\it lattice polytope associated with lattice presentations} $\Delta$ and $\Delta'$, and denote it by $P(\Delta, \Delta')$. Note that we have a choice of $P$. 

We denote by $\mathrm{Ver}_0(P)$ (respectively $\mathrm{Ver}_1(P)$) the vertices of $P$ coming from $\Delta$ (respectively $\Delta'$), which consist of a half of the elements of $\Delta$ (respectively $\Delta'$). 
We give an orientation of $P$ by giving each edge in the $x$-direction (respectively in the $y$-direction) the orientation from a vertex of $\mathrm{Ver}_0(P)$ to a vertex of $\mathrm{Ver}_1(P)$ (respectively from a vertex of $\mathrm{Ver}_1(P)$ to a vertex of $\mathrm{Ver}_0(P)$). 
\end{definition}

\begin{figure}[ht]
\centering
\includegraphics*[height=4cm]{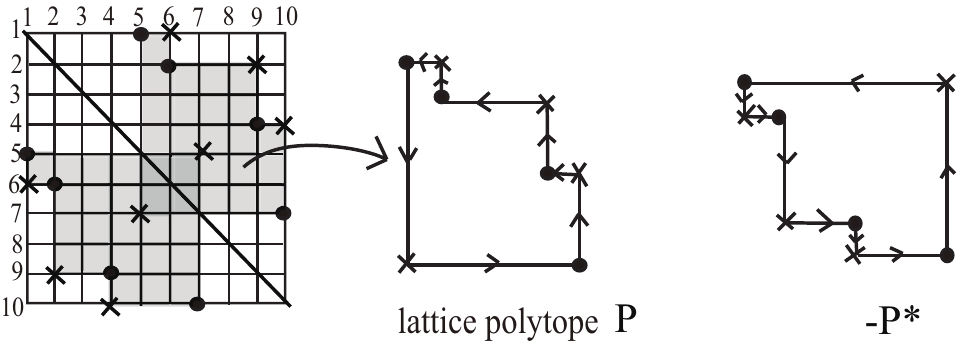}
\caption{The lattice polytope associated with $\Delta$ and $\Delta'$, where $\Delta$ is presented by black circles and $\Delta'$ is presented by X marks.}
\label{fig3-1}
\end{figure}

\begin{proof}[Proof of Proposition \ref{prop1}]
We take the set of points $\Delta$ and $\Delta'$ as the set of vertices. 
Since the set of $x$, $y$-components is a set of $2n$ distinct points in $\mathbb{R}$, for each point $v$ of $\Delta$, $v$ is either an isolated vertex satisfying $v=v'$ for $v' \in \Delta'$, or there are exactly two points $v'$ and $w'$ of $\Delta'$ such that $\overline{vv'}$ and $\overline{vw'}$ is in the $x$-direction and the $y$-direction, respectively. The same argument implies the same thing for each point of $\Delta'$. Thus we have a pair of lattice polytopes $P, -P^*$ satisfying the required conditions. That $-P^*$ is the mirror reflection of $P$ with respect to the line $x=y$ follows from the fact that 
for each point of $\Delta$ or $\Delta'$, its mirror reflection is also a point of $\Delta$ or $\Delta'$.  
\end{proof}

Since the sets of $x$, $y$-components of vertices of $P$ and $-P^*$ are the same, and the union of the vertices of $P\cup (-P^*)$ form $\Delta \cup \Delta'$, the set of $x$, $y$-components of vertices of $P$ is the same with that of $\Delta$ and $\Delta'$. Since each edge of a lattice polytope is either in the $x$-direction or the $y$-direction, the sets of $x$, $y$-components of $\mathrm{Ver}_i(P)$ ($i=0,1$) is the same. Thus the set of $x$, $y$-components of $\mathrm{Ver}_i(P)$ ($i=0,1$) is the same with that of $\Delta$ and $\Delta'$, and we have the following.

\begin{remark}\label{rem3-3}
Let $n$ be the number of the vertices of $\mathrm{Ver}_i(P)$ ($i=0,1$). 
Then, the set $X$ of $x$, $y$-components of $\mathrm{Ver}_i(P)$ ($i=0,1$) is the same with that of $\Delta$ and $\Delta'$, and it consists of $2n$ distinct elements.  
Thus the multiset of $x$, $y$-components of 
the vertices of $P$ is the set of two copies of $X$. 

Conversely, for a pair of sets $V_0$ and $V_1$ of points of $\mathbb{R}^2$ with the same set of $x$, $y$-components consisting of $2n$ distinct elements of $\mathbb{R}$, 
there is a unique lattice polytope $P$ such that $\mathrm{Ver}_i(P)=V_i$ $(i=0,1)$. 
\end{remark}

\begin{remark}
For a lattice polytope $P$, we denote by $-P$ the lattice polytope obtained from $P$ by orientation-reversal. Then, for a lattice polytope $P=P(\Delta, \Delta')$ associated with lattice presentations $\Delta$ and $\Delta'$,  
$-P(\Delta, \Delta')=P(\Delta', \Delta)$. 
Further, since we equipped $P$ with orientation such that each edge in the $x$-direction (respectively in the $y$-direction) is oriented  from a vertex of $\mathrm{Ver}_0(P)$ to a vertex of $\mathrm{Ver}_1(P)$ (respectively from a vertex of $\mathrm{Ver}_1(P)$ to a vertex of $\mathrm{Ver}_0(P)$), the orientation changes when we rotate $P$ by $\pi/2$:  
the $\pi/2$-rotation of unoriented $P$ is as a new polytope $-\rho(P)$, where $\rho(P)$ is the $\pi/2$-rotation of oriented $P$. Similarly,  
the mirror reflection of unoriented $P$ is as a new polytope $-P^*$, the orientation-reversed mirror reflection of $P$. 
\end{remark}

We introduce the equivalence relation among lattice polytopes.  
Let $P$ be a lattice polytope with $2n$ vertices. 
Since the set of $x$, $y$-components of $\mathrm{Ver}_i(P)$ 
($i=0,1$) consists of $2n$ distinct elements, for  $\mathrm{Ver}_i(P)=\{v_1, \ldots, v_n\}$ with $v_j=(x_j, y_j)$ ($j =1,2,\ldots,n$)  satisfying $x_1<x_2 <\ldots<x_n$, 
we take an element $\sigma$ of the symmetric group $\mathfrak{S}_n$ of $n$ elements determined by 
$y_{\sigma^{-1}(1)}<y_{\sigma^{-1}(2)}<\ldots<y_{\sigma^{-1}(n)}$. We denote the element $\sigma$ of $\mathfrak{S}_n$ by $\sigma_i(P)$ $(i=0,1)$. 
Let $\pi$ be a permutation \[
\pi=\begin{pmatrix} 1 & 2 & \ldots & n-1& n \\
n & n-1 & \ldots& 2 & 1\end{pmatrix} \in \mathfrak{S}_n. \]

\begin{proposition}\label{prop3-4}
Let $\sigma$ be an element of $\mathfrak{S}_n$, which will be  identified with a set of points of $\mathbb{R}^2$, $\{ (j,\sigma(j))-(n/2, n/2) \mid j=1,2,\ldots,n\}$. Then, the mirror reflection of $\sigma$ with respect to the line $x=n+1$ is $\pi\sigma$, and the $\pi/2$ rotation of $\sigma$ is $\sigma^{-1}\pi$.
\end{proposition}

\begin{definition}
Let $P, P'$ be lattice polytopes 
consisting of $2n$ vertices. 
Then, we say $P$ and $P'$ are {\it equivalent} if $(\sigma_0(P), \sigma_1(P))$ and $(\sigma_0(P'), \sigma_1(P'))$, or $(\sigma_0(P), \sigma_1(P))$ and $((\sigma_0(P'))^{-1}, (\sigma_1(P'))^{-1})$, are related by right or left multiplication by the permutation $\pi$.
\end{definition}

\begin{proof}[Proof of Proposition \ref{prop3-4}]
Let $S$ be a matrix presenting $\sigma$. 
The matrix presenting $\pi$, which will be denoted by $P$, is 
\[P=\begin{pmatrix} 0  & \cdots &0 & 1\\
0  & \cdots & 1 & 0\\
 \vdots & \ddots & \vdots\\
1  & \cdots & 0 &0\end{pmatrix}.
\]
Since the right multiplication by $P$ changes the $j$th column of the operated matrix to the $(n-j)$th column $(j=1,2,\ldots,n)$, $SP$  presents the mirror reflection of $\sigma$ with respect to the line $x=n+1$. Since the order of a product of matrices is the reversed order of a product of the presented elements of $\mathfrak{S}_n$, we see that $\pi \sigma$ is the mirror reflection of $\sigma$ with respect to the line $x=n+1$. 

Let $\{(x_j, y_j) \mid j=1,2,\ldots,n\}$ ($x_1<x_2<\ldots<x_n$) be the set of points of $\mathbb{R}^2$ identified with $\sigma$. Then, we have $y_{\sigma^{-1}(1)}<y_{\sigma^{-1}(2)}<\ldots<y_{\sigma^{-1}(n)}$. Let us rotate $\sigma$ by $\pi/2$. Then, the set of points changes to $\{(-y_j, x_j) \mid j=1,2,\ldots,n\}$. 
Put $x_j'=-y_{\sigma^{-1}(n+1-j)}$ and $y_j'=x_{\sigma^{-1}(n+1-j)}$. Let $\rho$ be the element of $\mathfrak{S}_n$ presenting the $\pi/2$ rotation of $\sigma$. Since $-y_{\sigma^{-1}(n)}<-y_{\sigma^{-1}(n-1)}<\ldots<-y_{\sigma^{-1}(1)}$, $x'_1<x'_2<\ldots<x'_n$, hence we see that  $y'_{\rho^{-1}(1)}<y'_{\rho^{-1}(2)}<\ldots<y'_{\rho^{-1}(n)}$. Since $y_j'=x_{\sigma^{-1}(n+1-j)}=x_{\sigma^{-1}\pi(j)}$, $x_j=y'_{\pi \sigma(j)}$ ($j=1,2,\ldots,n$). Thus $\rho^{-1}=\pi \sigma$, and we see $\rho=\sigma^{-1}\pi$.  
\end{proof}

In particular, by Proposition \ref{prop3-4} and by definition, we have the following. 
For a lattice polytope $P$, we call a subgraph of $P$ whose boundary is homeomorphic to an immersed circle a {\it component} of $P$, and 
we say a lattice polytope is {\it connected} if it consists of one component. 

\begin{proposition}
For a lattice polytope $P$,  
$P \sim -P ^*$. 
Thus, a connected lattice polytope $P$ associated with lattice presentations of partial matchings is unique up to equivalence. 
\end{proposition} 

\section{Describing transformations of partial matchings by lattice polytopes}\label{sec4}

In this section, we consider transformations of lattice presentations of partial matchings. 
In Section \ref{sec4-1}, we introduce transformations of lattice presentations and lattice polytopes. In Section \ref{sec4-2}, we introduce the area of a transformation and we give a lower estimate of the area of a transformation and in certain cases construct a transformation with minimal area (Theorem \ref{thm3-10}). Lemmas and Propositions are shown in Section \ref{sec4-3}. 

\subsection{Transformations of lattice presentations and lattice polytopes}\label{sec4-1}
For two arcs $a_1=(x_1, y_1)$ and $a_2=(x_2, y_2)$ of a chord diagram, we consider new arcs $a_1'=(x_1', y_1')$ and $a_2'=(x_2', y_2')$, where $\{x_1', x_2', y_1', y_2'\}=\{x_1, x_2, y_1, y_2\}$ with $\{a_1, a_2\} \neq \{a_1', a_2'\}$. We consider a new chord diagram obtained from exchanging $a_1, a_2$ to $a_1', a_2'$. We call the new chord diagram the result of a {\it transformation} between two arcs $a_1$ and $a_2$.

For a lattice presentation of a chord diagram, we define 
a {\it transformation} as the operation presenting a transformation of the chord diagram. 
For two lattice presentations $\Delta$ and $\Delta'$ with the same set of $x, y$-components, we define a {\it transformation from $\Delta$ to $\Delta'$} as a sequence of transformations  such that the initial and the terminal lattice presentations are $\Delta$ and $\Delta'$, respectively. We denote it by $\Delta \to \Delta'$. 

\begin{definition}
Let $\Delta$ be a set of lattice points. For a point $u$ of $\mathbb{R}^2$, we denote by $x(u)$ and $y(u)$ the $x$ and $y$-components of $u$ respectively.  For two distinct points $v, w$ of $\Delta$, 
we consider the rectangle $R(v, w)$ one pair of whose diagonal vertices are $v$ and $w$. Put $\tilde{v}=(x(w), y(v))$ and $\tilde{w}=(x(v), y(w))$, which form the other pair of diagonal vertices of $R(v,w)$. Then, consider a new set of lattice points obtained from $\Delta$, from removing $v,w$ and adding $\tilde{v}, \tilde{w}$. 
We call the new set of lattice points the result of a {\it transformation of $\Delta$ by the rectangle} $R(v, w)$, and denote it by $t( \Delta, R(v, w))$. 

For two sets of lattice points $\Delta$ and $\Delta'$ with the same set of $x, y$-components, we define a {\it transformation from $\Delta$ to $\Delta'$} as a sequence of transformations by rectangles such that the initial and the terminal lattice points are $\Delta$ and $\Delta'$  respectively. We will denote it by $\Delta \to \Delta'$. 
\end{definition}

Recall that for a lattice polytope $P$, we denote by $-P^*$ the orientation-reversed mirror reflection of $P$ with respect to the line $x=y$.

\begin{lemma}\label{lem4-2}
A transformation between two arcs of a chord diagram is presented by a transformation of the presenting lattice presentation $\Delta$ by rectangles $R(v,w)$ and $-R(v,w)^*$  
for $v,w \in \Delta$. 
\end{lemma}

\begin{proof}
Since two arcs of a chord diagram are either separated, nesting, or crossing, 
it suffices to consider the following three cases: (1) $a_1, a_2$ are separated and $a_1', a_2'$ are nesting, (2) $a_1, a_2$ are nesting and $a_1', a_2'$ are crossing, and (3) $a_1, a_2$ are crossing and $a_1', a_2'$ are separated. 
Since the transformation is described as in Figure \ref{fig4-1}, we have the required result. 
\end{proof}

\begin{figure}
\centering
\includegraphics*[height=8cm]{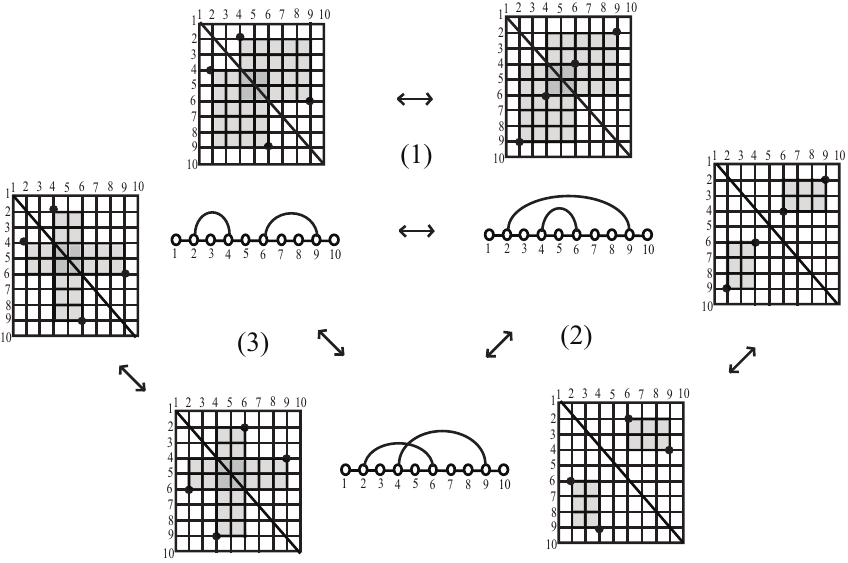}
\caption{Transformations between two arcs and the presenting transformations of  lattice presentations.}
\label{fig4-1}
\end{figure}

Let $P$ be a lattice polytope.  
Then, for $v, w \in \mathrm{Ver}_0(P)$, consider a new set of vertices obtained from $\mathrm{Ver}_0(P)$ by a transformation by $R(v,w)$. 
Since the multiset of $x$, $y$-components are preserved, the  resulting new vertices 
and $\mathrm{Ver}_1(P)$ form a new lattice polytope (see Remark \ref{rem3-3}). 

\begin{definition}
For a lattice polytope $P$ and $v,w \in \mathrm{Ver}_0(P)$, 
we call the lattice polytope $P'$ determined by the lattice points $\mathrm{Ver}_0(P')=t(\mathrm{Ver}_0(P), R(v, w))$ and $\mathrm{Ver}_1(P')=\mathrm{Ver}_1(P)$ 
the result of a {\it transformation of $P$ by the rectangle} $R=R(v, w)$ and denote it by $t(P , R)$. 

For two lattice polytopes $P$ and $P'$ with the same set of $x, y$-components, we define a {\it transformation from $P$ to $P'$} as a sequence of transformations by rectangles $R_j$ such that the initial and the terminal lattice polytopes are $P$ and $P'$, respectively. We denote it by $P \to P'$. 
\end{definition}

For a pair of lattice polytopes $P$ and $-P^*$, and rectangles $R$ and $-R^*$, we denote $t(P , R ) \cup t(-P^*, -R^* )$ by $t(P\cup (-P^*) , R \cup (-R ^*))$, and call it the result of a {\it transformation of $P\cup (-P^*)$ by rectangles} $R \cup (-R ^*)$. 
Note that $t(-P^*, -R^*)=-t(P, R)^*$ and $t(P\cup (-P^*) , R \cup (-R ^*))=t(P , R ) \cup (-t(P, R )^*)$. 

\begin{definition}
For two pairs of lattice polytopes $P, -P^*$ and $P', -P'^*$ with the same set of $x, y$-components, we define a {\it transformation from $P\cup (-P^*)$ to $P'\cup (-P'^*)$} as a sequence of transformations by rectangles $R_j \cup (-R_j^*)$ such that the initial and the terminal lattice polytopes are $P\cup (-P^*)$ and $P'\cup (-P'^*)$, respectively, and denote it by $P\cup (-P^*) \to P'\cup (-P'^*)$. 
\end{definition}

For lattice presentations of partial matchings $\Delta, \Delta'$ with the same set of $x, y$-components, and an associated lattice polytope $P$, 
$\Delta=\mathrm{Ver}_0(P \cup  (-P^*))$, $\Delta'=\mathrm{Ver}_1(P \cup  (-P^*))$. 
Thus, by Lemma \ref{lem4-2}, we have the following. 
When we consider transformations of lattice polytopes, we regard isolated vertices  $\mathrm{Ver}_1(Q)$ as a lattice polytope $Q'$ whose initial and terminal vertices are the isolated vertices: $\mathrm{Ver}_0(Q')=\mathrm{Ver}_1(Q')=\mathrm{Ver}_1(Q)$. 

\begin{proposition}
Let $\Delta$, $\Delta'$ be lattice presentations of partial matchings with the same set of $x, y$-components, and let $P$ be a lattice polytope associated with $\Delta$ and $\Delta'$. 
Then, a transformation $\Delta \to \Delta'$ is described by a transformation of lattice  
polytopes $P \cup (-P^*) \to \mathrm{Ver}_1(P \cup  (-P^*))$. 

In particular, a transformation of a lattice  
polytope $P \to \mathrm{Ver}_1(P)$ presents a transformation $\Delta \to \Delta'$. 
\end{proposition}

\begin{example}
We consider lattice presentations $\Delta$ and $\Delta'$ as in Figure \ref{fig3-1}, where we denote the points of $\Delta=\mathrm{Ver}_0(P \cup  (-P^*))$ by black circles and the points of $\Delta'=\mathrm{Ver}_1(P  \cup  (-P^*))$ by X marks.  In Figure \ref{fig4-2}, we give an example of a transformation $\Delta \to \Delta'$, described by a transformation of lattice  
polytopes $P \cup (-P^*) \to \mathrm{Ver}_1(P \cup  (-P^*))$, where we denote by shadowed rectangles the used rectangles. In this case, it suffices to see a transformation $P \to \mathrm{Ver}_1(P)$, which induces a transformation $P \cup (-P^*) \to \mathrm{Ver}_1(P \cup  (-P^*))$. 
 
\begin{figure}[ht]
\centering
\includegraphics*[height=6cm]{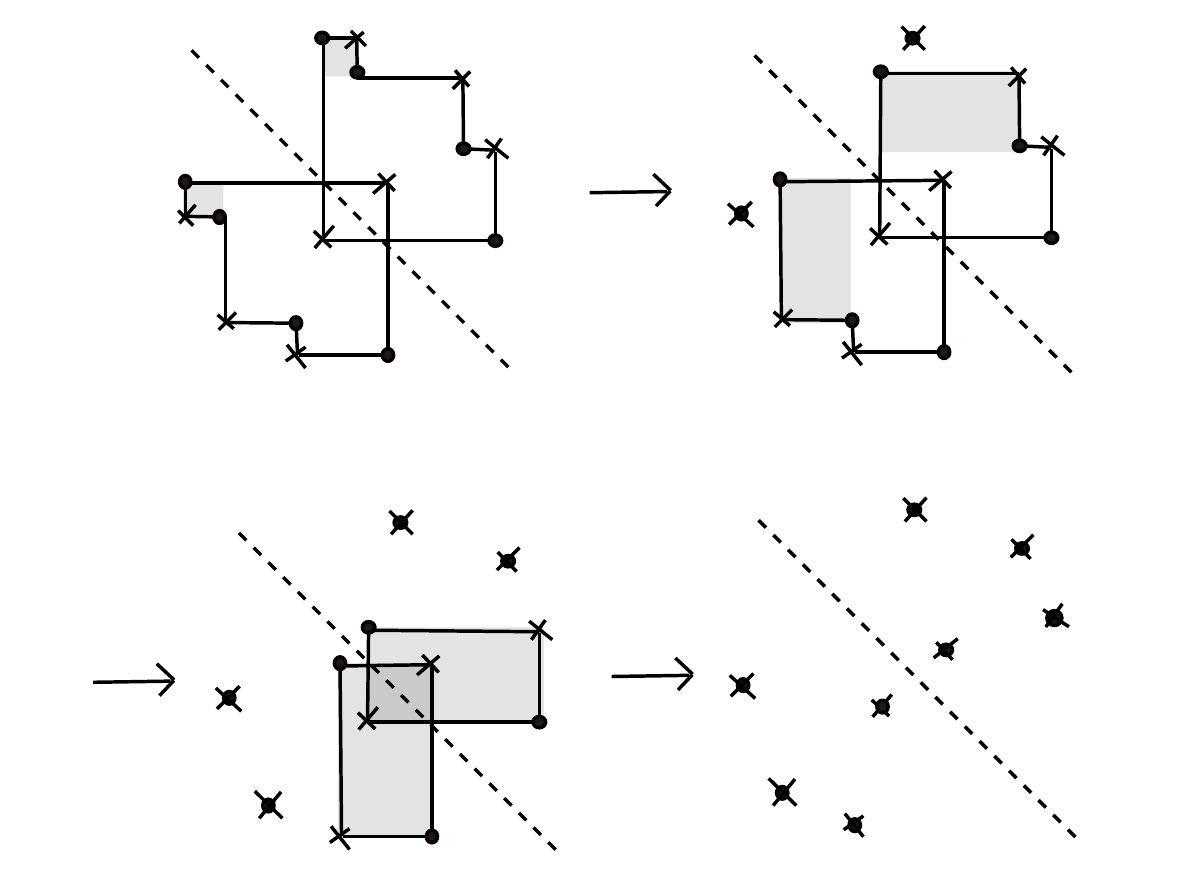}
\caption{Transformation of lattice polytopes $P \cup (-P^*)$ associated with lattice presentations of partial matchings, where the vertices of $\mathrm{Ver}_0(P\cup (-P^*))$ (respectively $\mathrm{Ver}_1(P\cup (-P^*))$) are indicated by black circles (respectively X marks)}
\label{fig4-2}
\end{figure}
\end{example}

\subsection{Areas of transformations}\label{sec4-2}

For a lattice polytope $P$, let us define the area of $|P|$, $\mathrm{Area}|P|$, as follows. 
Recall that for a lattice polytope $P$, 
we give an orientation of $P$ by giving each edge in the $x$-direction (respectively in the $y$-direction) the orientation from a vertex of $\mathrm{Ver}_0(P)$ to a vertex of $\mathrm{Ver}_1(P)$ (respectively from a vertex of $\mathrm{Ver}_1(P)$ to a vertex of $\mathrm{Ver}_0(P)$). Then, $\partial P$ has a coherent orientation as an immersion of a union of several circles. 
The space $\mathbb{R}^2$, which contains $P$, is divided into several regions $A_1,\ldots, A_m$ by $\partial P$. For each region $A_i$ ($i=1,\dots,m$), let $\omega(A_i)$ be the {\it rotation number} of $P$ with respect to $A_i$, which is the sum of rotation numbers of the components 
of $P$, with respect to $A_i$. Here, the {\it rotation number} of a connected lattice polytope $Q$ with respect to a region $A$ of $\mathbb{R}^2 \backslash\partial Q$ is the rotation number of a map $f$ from $\partial Q=\mathbb{R}/2\pi \mathbb{Z}$ to $\mathbb{R}/2\pi\mathbb{Z}$ which maps $x \in \partial Q$ to the argument of the vector from a fixed interior point of $A$ to $x$. Here, the {\it rotation number} of the map $f$ is defined by 
$(F(x)-x)/2\pi$, where $F: \mathbb{R} \to \mathbb{R}$ is the lift of $f$ and $x \in \partial Q$.
We define $\mathrm{Area}(A_i)$ for a region $A_i$ by the area induced from $\mathrm{Area}(R)=|x_2-x_1||y_2-y_1|$ for a rectangle $R$ whose diagonal vertices are $(x_1, y_1)$ and  $(x_2, y_2)$. 
Then, we define the {\it area} of $P$, denoted by $\mathrm{Area}(P)$, by $\mathrm{Area}(P)=\sum_{i=1}^m \omega(A_i)\mathrm{Area}(A_i)$, and the {\it area} of $|P|$, denoted by $\mathrm{Area}|P|$, by $\mathrm{Area}|P|=\sum_{i=1}^m |\omega(A_i)\mathrm{Area}(A_i)|$. 
\begin{remark}
For two lattice polytopes $P_1$ and $P_2$, $\mathrm{Area}|P_1 \cup P_2|=\mathrm{Area}|(-P_1^*)\cup (-P_2^*)|$, but $\mathrm{Area}|P_1\cup P_2|$ is not always equal to $\mathrm{Area}|P_1\cup (-P_2^*)|$. 
Thus, for a lattice polytope $P$ associated with lattice presentations $\Delta$ and $\Delta'$, $\mathrm{Area}|P|$ depends on the choice of the components of $P$.  
\end{remark}
 
\begin{definition}
Let $\Delta, \Delta'$ be two lattice presentations of partial matchings with the same set of $x, y$-components. 
Let us consider a transformation 
$\Delta=\Delta_0 \to \Delta_1\to \cdots \to \Delta_k=\Delta'$,  with $\Delta_j=t(\Delta_{j-1}, R_j\cup (-R_j^*))$ for a rectangle $R_j$ ($j=1,2,\ldots,k$). 
Then, we call $\sum_{j=1}^k |\mathrm{Area}(R_j)|$ the {\it area of a transformation $\Delta \to \Delta'$}.
\end{definition} 

We say that a connected lattice polytope $P$ is {\it simple} if the face of $P$ is homeomorphic to a 2-disk in $\mathbb{R}^2$. 
For a lattice polytope $P$ with regions $A_1, \ldots, A_m$ divided by $\partial P$, 
attach each region $A_i$ with right-handed (resp. left-handed) orientation when $\omega(A_i)$ is positive (resp. negative) $(i=1,\ldots,m)$. Then the union of $|\omega(A_i)|$ copies of the closure of $A_i$ bounds $\partial P$. We call the union the {\it region} of $P$, which will be denoted by the same notation $P$. 
Further, if edges $\overline{uv}$ and $\overline{wz}$ of $P$ have a transverse intersection, then we call the intersection point a {\it crossing}. 

\begin{theorem}\label{thm3-10}
Let $\Delta, \Delta'$ be two lattice presentations of partial matchings with the same set of $x, y$-components. 
We consider a transformation 
$\Delta=\Delta_0 \to \Delta_1\to \cdots \to \Delta_k=\Delta'$,  with $\Delta_j=t(\Delta_{j-1}, R_j\cup R_j^*)$ for a rectangle $R_j$ $(j=1,2,\ldots,k)$. Let $P$ be a lattice polytope associated with $\Delta$ and $\Delta'$. 
Then
\begin{equation}\label{eq0}
\sum _{j=1}^k \mathrm{Area}(R_j) =\frac{1}{2}\mathrm{Area}(P\cup (-P^*)) 
\end{equation}
and 
\begin{equation}\label{eq1}
\sum_{j=1}^k |\mathrm{Area}(R_j)|\geq \frac{1}{2}\mathrm{Area}|P\cup (-P^*)|. 
\end{equation}

Further, when $P$ satisfies either the following condition $(1)$ or $(2)$, 
\begin{equation}\label{eq2}
\sum_{j=1}^k |\mathrm{Area}(R_j)| \geq \mathrm{Area}|P|, 
\end{equation}
and there exists transformations which realize the equality of $(\ref{eq2})$, where the conditions are as follows. 

\begin{enumerate}[$(1)$]
\item
The rotation number $\omega(A)$ is equal to $\epsilon$ for any region $A$ surrounded by an embedded closed path in $\partial P$, where $\epsilon\in \{+1, -1\}$  (see Figure \ref{fig4-3} for example).
\item
The lattice polytopes $P$ and $-P^*$ are disjoint, and, 
when we regard crossings as vertices, $P$ is regarded as the union of simple lattice polytopes $P_1, \ldots, P_m$ and $Q_{i1}, \ldots, Q_{in_i}$ $(i=1,\ldots,m)$ such that $P_i \cap P_j$ $(i \neq j)$ is empty or consists of one crossing in $\partial P_i \cap \partial P_j$ and $Q_{i1}, \ldots, Q_{in_i}$ are mutually disjoint and contained in the interior of $P_i$ (see Figure $\ref{fig4-4}$ for example).
\end{enumerate}
 \end{theorem}

\begin{figure}[ht]
\centering
\includegraphics*[height=3cm]{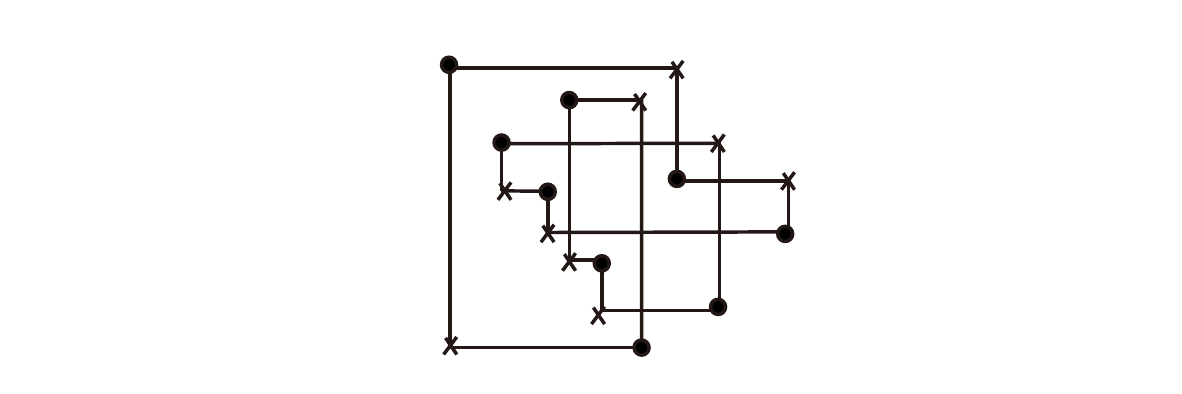}
\caption{Example of a lattice polytope $P$ satisfying the condition (1), where the vertices of $\mathrm{Ver}_0(P)$ (respectively $\mathrm{Ver}_1(P)$) are indicated by black circles (respectively X marks).}
\label{fig4-3}
\end{figure}

\begin{figure}[ht]
\centering
\includegraphics*[height=3.5cm]{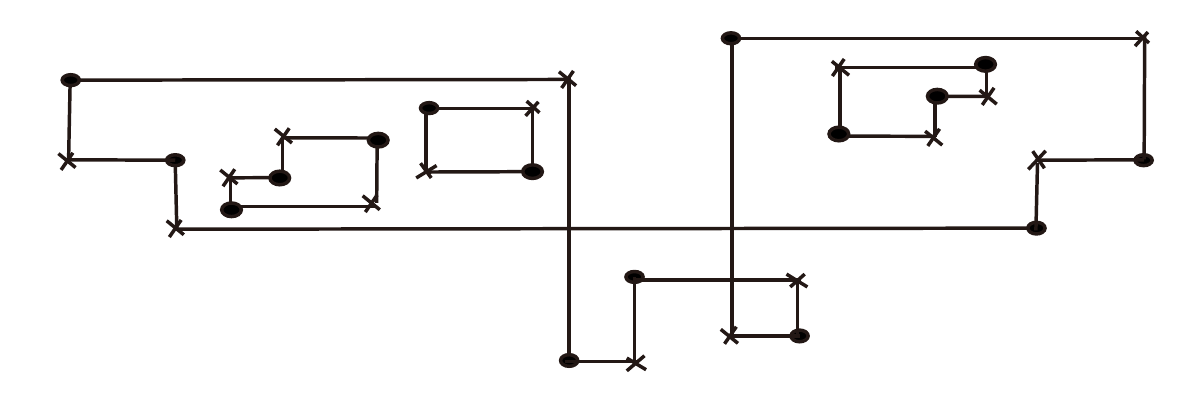}
\caption{Example of a lattice polytope $P$ satisfying the condition (2), where the vertices of $\mathrm{Ver}_0(P)$ (respectively $\mathrm{Ver}_1(P)$) are indicated by black circles (respectively X marks).}
\label{fig4-4}
\end{figure}

\begin{remark}
The condition (1) indicates that each connected component of $P$ is in the form of the projected image into the $xy$-plane of a closed braid in the $xyz$-space in general position with respect to the $z$-axis \cite{Birman}. 
\end{remark}

\begin{proof}[Proof of Theorem \ref{thm3-10}]
For a point $u$ of $\mathbb{R}^2$, we denote by $x(u)$ and $y(u)$ the $x$ and $y$-components of $u$ respectively. 
For a rectangular $R(v,w)$, 
we consider the other pair of diagonal vertices of $R(v,w)$,  
$\tilde{v}=(x(w), y(v))$ and $\tilde{w}=(x(v), y(w))$. 
Then, we assign to $R(v,w)$ an orientation such that the edges are oriented from $v$ to $\tilde{v}$, from $\tilde{v}$ to $w$, from $w$ to $\tilde{w}$, and from $\tilde{w}$ to $v$. 
Then, this induces an orientation of $P$ and $-P^*$ which coincides with the original orientation of $P$ and $-P^*$, 
and by Lemma \ref{lem4-15} we see that 
\[
\sum _{j=1}^k (\mathrm{Area}(R_j \cup (-R_j^*))=\mathrm{Area}(P \cup (-P^*)), 
\]
and 
\[
\sum _{j=1}^k \mathrm{Area}|R_j \cup R_j^*|\geq\mathrm{Area}|P \cup (-P^*)|. 
\]
Let $Q$ be a lattice polytope. Then, since the mirror reflection of an edge of $Q$ in the $x$-direction (respectively $y$-direction) is an edge of $Q^*$ in the $y$-direction (respectively $x$-direction), the orientation of any closed path $C$ in $\partial Q$ and $C^*$ in $\partial Q^*$ coincide as an embedded circle in $\mathbb{R}^2$. Hence 
$\mathrm{Area}(Q)=\mathrm{Area}(Q^*)$, and hence, for $j=1,\ldots, k$, 
\begin{equation*} \mathrm{Area}(R_j\cup (-R_j^*))
=  \mathrm{Area}(R_j)+\mathrm{Area}(-R_j^*)=2  \mathrm{Area}(R_j)
\end{equation*}
and 
\begin{equation*} \mathrm{Area}|R_j\cup (-R_j^*)|
= | \mathrm{Area}(R_j)|+|\mathrm{Area}(-R_j^*)|=2 |\mathrm{Area}(R_j)|.  
\end{equation*}
Hence we have (\ref{eq0}) and (\ref{eq1}). 

When the lattice polytope $P$ satisfies the condition 1 or 2, $\mathrm{Area}(P \cup (-P^*))=\mathrm{Area}(P)+\mathrm{Area}(-P^*)$. Since $\mathrm{Area}(P)=\mathrm{Area}(-P^*)$, we have (\ref{eq2}). In these cases, the boundary $\partial P$ can be divided to a union of boundaries of simple connected lattice polytopes. 
In order to show that there exists a transformation which realizes the equality of (\ref{eq2}),  
first we show the following claims. 

Claim (a). For a simple connected lattice polytope $Q$, there is a transformation $\mathrm{Ver}_0(Q) \to \mathrm{Ver}_1(Q)$ by rectangles $R_j$ $(j=1,\ldots, k)$ such that 
\begin{equation}\label{eq3-2}
\sum _{j=1}^k |\mathrm{Area}(R_j)|=|\mathrm{Area}(Q)|. 
\end{equation}

Claim (b). 
For simple connected lattice polytopes $Q, Q_1, \ldots, Q_l$ such that $Q_1, \ldots, Q_l$ are mutually disjoint and contained in the interior of $Q$, and $Q \cup Q_1 \cup \cdots \cup Q_l$ is a region obtained from a 2-disk $Q$ by removing $Q_1,\ldots, Q_l$, there is a transformation $\mathrm{Ver}_0(Q \cup \cup_{i=1}^l Q_i) \to \mathrm{Ver}_1(Q\cup \cup_{i=1}^l Q_i)$ by rectangles $R_j$ $(j=1,\ldots,k)$ such that 
\begin{equation}\label{eq3-3}
\sum _{j=1}^k |\mathrm{Area}(R_j)|=\mathrm{Area}|Q \cup \cup_{i=1}^lQ_i|=\big||\mathrm{Area}(Q)|-\sum_{i=1}^l|\mathrm{Area}(Q_i)|\big|.  
\end{equation} 

First we show Claim (a). 
By Proposition \ref{lem3-15}, there is a rectangular $R=R(v,w)$ contained in the region $P$ with $v,w \in \mathrm{Ver}_0(Q)$ such that $Q\backslash R$ forms a simple lattice polytope $Q'$. 
This implies that  
$|\mathrm{Area}(R)|+|\mathrm{Area}(Q')|=|\mathrm{Area}(Q)|$. 
Hence, in order to show Claim (a), it suffices to show that there is a transformation $\mathrm{Ver}_0(Q') \to \mathrm{Ver}_1(Q')$ satisfying (\ref{eq3-2}) with $Q=Q'$.  Let $n$ be the number of vertices of $\mathrm{Ver}_0(Q)$. Since $v,w$ of $R=R(v,w)$ are vertices of $\mathrm{Ver}_0(Q)$, we see that if $Q'$ is connected, then the number of vertices of $\mathrm{Ver}_0(Q')$ is $n-1$. If $Q'$ is not connected, then $Q'$ consists of 2 components, and the sum of the number of vertices of $\mathrm{Ver}_0(Q')$ is $n$, thus the number of vertices of $\mathrm{Ver}_0(Q')$ which come from a connected component is less than $n$. If the number of vertices of $\mathrm{Ver}_0(Q')$ is 2, then $Q'$ is a rectangular, and we have (\ref{eq3-2}) with $Q=Q'$. Thus, by induction on the number of vertices of connected components, we can construct a transformation satisfying (\ref{eq3-2}).  

For the other Claim (b), by a similar argument, we can show that there exists a transformation satisfying (\ref{eq3-3}). We can take rectangles $R$ inductively. When the vertices $v,w$ of a used rectangle $R(v,w)$ come from distinct components, 
then the number $l$ of simple connected lattice polytopes is reduced by one. 

Now, we construct a transformation such that $\sum_j |\mathrm{Area}(R_j)|=\mathrm{Area}|P|$, assuming that $P$ satisfies the condition (1) or (2). 
We show the case when $P$ is connected, and $Q_{i1}, \ldots, Q_{in_i}$ $(i=1,\ldots, m)$ of the condition (2) are empty graphs. 
Let us regard $\partial P$ as an immersion of a circle. 

(Step 1) 
Take an interval $I$ of $\partial P$ which starts from a crossing $x$ and comes back to $x$. If there is another interval $J$ in $I$ which starts from a crossing $y$ and comes back to $y$, then take $J$ instead of $I$. Repeating this process, we have an interval $I$ of $\partial P$ which bounds a simple connected lattice polytope $Q_1$ whose vertices consists of the vertices of $P$ and one crossing $x$. 

(Step 2) For a lattice polytope $Q_1$ of Step 1, 
if $x \in \mathrm{Ver}_1(Q_1)$, then put $P_1=Q_1$. Then, by Claim (a), we have a transformation  
$P \to (P\backslash P_1)\cup \mathrm{Ver}_1(P_1)$. 
If $x \in \mathrm{Ver}_0(Q_1)$, then consider $Q_2=P\backslash Q_1$, and repeat Step 1 several times. Since $x \in \mathrm{Ver}_0(Q_1)$ become a vertex in $\mathrm{Ver}_1(Q_2)$, we obtain $Q_n$ such that the crossings are in $\mathrm{Ver}_1(Q_n)$. Put $P_1=Q_n$, and then by Claim (a) we see that there is a transformation $P \to (P\backslash P_1) \cup \mathrm{Ver}_1(P_1)$. 

By repeating this process, we have a transformation $P \to \mathrm{Ver}_1(P)$, which is a transformation $ \Delta  \to  \Delta'$ such that $\sum_j |\mathrm{Area}(R_j)|=\mathrm{Area}|P|$. See Figures \ref{fig4-5} and \ref{fig4-6}. 

The case when  $Q_{i1}, \ldots, Q_{in_i}$ $(i=1,\ldots, m)$ of  the condition 2 may not be empty graphs can be shown by the same argument by using Claim (b). Thus there is a transformation such that $\sum_j |\mathrm{Area}(R_j)|=\mathrm{Area}|P|$ when $P$ satisfies the condition (1) or (2). 
\end{proof}

\begin{figure}[ht]
\centering
\includegraphics*[height=6cm]{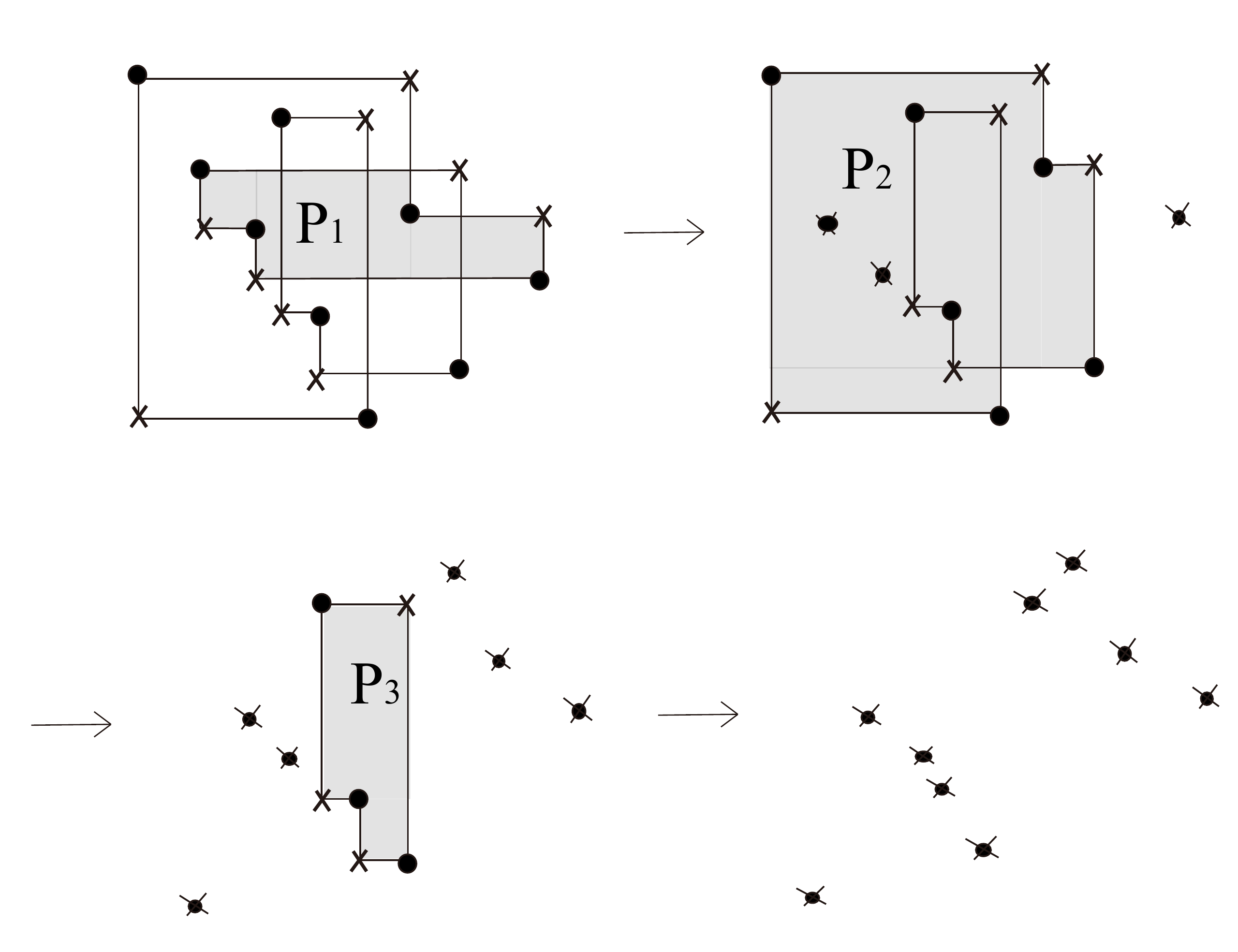}
\caption{Example of a transformation of a lattice polytope $P$ satisfying the condition (1). We have a transformation of $P$ by composite of transformations of $P_1, P_2, P_3$, where $P_1, P_2, P_3$ are indicated by shadowed regions, and the vertices of $\mathrm{Ver}_0(P)$ (respectively $\mathrm{Ver}_1(P)$) are indicated by black circles (respectively X marks). }
\label{fig4-5}
\end{figure}

\begin{figure}[ht]
\centering
\includegraphics*[height=3cm]{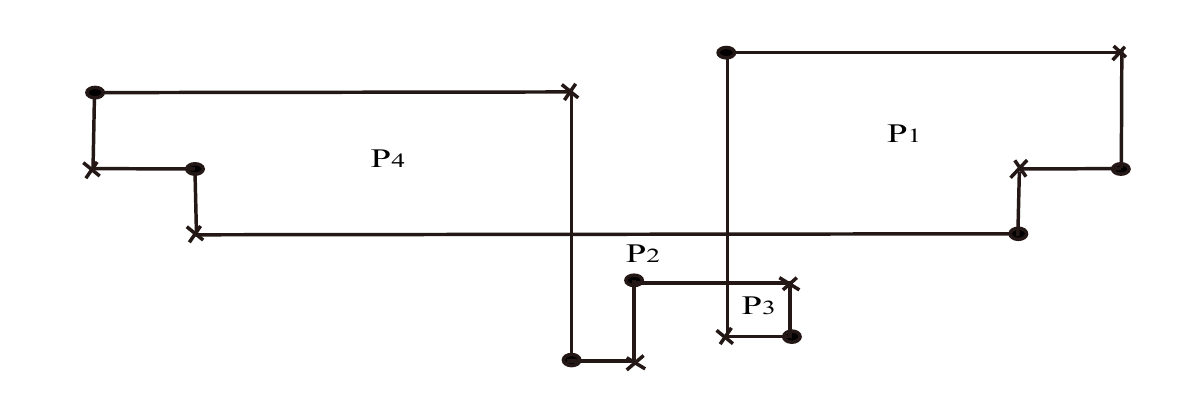}
\caption{Example of a transformation of a lattice polytope $P$ satisfying the condition (2). We have a transformation of $P$ by composite of transformations of $P_1, P_2, P_3, P_4$, where the vertices of $\mathrm{Ver}_0(P)$ (respectively $\mathrm{Ver}_1(P)$) are indicated by black circles (respectively X marks). }
\label{fig4-6}
\end{figure}
  
\begin{proposition}\label{prop3-18}
There exists a lattice polytope $P$ such that for any transformation of $P$ by  rectangles $R_1, \ldots, R_k$, $\sum_{j=1}^k|\mathrm{Area}(R_j)|>\mathrm{Area}|P|$. 
\end{proposition}

\begin{proof}
We consider a lattice polytope as illustrated by the leftmost figure of Figure \ref{fig4-7}, where the vertices of $\mathrm{Ver}_0(P)$ (respectively $\mathrm{Ver}_1(P)$) are indicated by black circles (respectively X marks), and the numbers in regions divided by $\partial P$ denote the rotation numbers. 
Assume that there is a transformation of $P$ by rectangles $R_1, \ldots, R_k$ such that $\sum_{j=1}^k|\mathrm{Area}(R_j)|=\mathrm{Area}|P|$. Then, by Lemma \ref{lem4-15}, any $R_j$ is disjoint with any region whose rotation number is zero. Thus, we have only one applicable rectangle $R_1=R(v_1, w_1)$ for $v_1,w_1 \in \mathrm{Ver}_0(P)$ as in the figure. Then, by a transformation of $P$ by $R_1$, we have a lattice polytope $P_1$ as in the middle figure of Figure \ref{fig4-7}. By the same argument, we have only one applicable rectangle $R_2=R(v_2, w_2)$ for $v_2,w_2 \in \mathrm{Ver}_0(P_1)$ as in the figure. Then, by a transformation of $P_1$ by $R_2$, we have a lattice polytope $P_2$ as in the right figure of Figure \ref{fig4-7}. Now, ignoring the isolated vertex, every possible rectangle of $P_2$ has intersection with a region with the rotation number zero. This implies that there is not a transformation of $P$ by rectangles $R_1, \ldots, R_k$ such that $\sum_{j=1}^k|\mathrm{Area}(R_j)|=\mathrm{Area}|P|$, and the required result follows.  
\end{proof}

\begin{figure}[ht]
\centering
\includegraphics*[height=3.5cm]{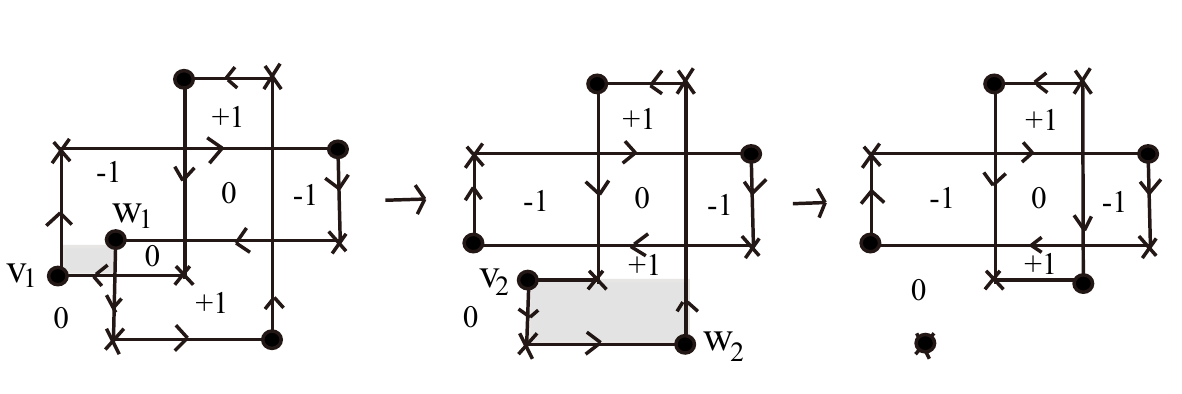}
\caption{Example of a transformation of a lattice polytope $P$ satisfying $\sum_{j=1}^k|\mathrm{Area}(R_j)|>\mathrm{Area}|P|$, where the vertices of $\mathrm{Ver}_0(P)$ (respectively $\mathrm{Ver}_1(P)$) are indicated by black circles (respectively X marks), and the numbers in regions divided by $\partial P$ denote the rotation numbers. }
\label{fig4-7}
\end{figure}

\begin{corollary}
There exist lattice presentations $\Delta, \Delta'$ such that any transformation $\Delta \to \Delta'$ satisfies $\sum_{j=1}^k(|\mathrm{Area}(R_j)|+|\mathrm{Area}(-R_j^*)|)>\frac{1}{2}\mathrm{Area}|P\cup (-P^*)|$, where $R_1, \ldots, R_k$ are used rectangles and $P$ is a lattice polytope associated with $\Delta, \Delta'$. Thus, there exist lattice presentations $\Delta, \Delta'$ such that the minimal area of transformations $\Delta \to \Delta'$ is greater than $\frac{1}{2}\mathrm{Area}|P\cup (-P^*)|$. 
\end{corollary}

\begin{proof}
Take lattice presentations $\Delta, \Delta'$ presented by lattice polytopes $P, -P^*$ such that $P\cap (-P^*)=\emptyset$ and $P$ is the lattice polytope given in Proposition \ref{prop3-18}. Then $\Delta, \Delta'$ are the required presentations. 
\end{proof}

The transformation with minimal area $\mathrm{Area}|P|$ which we constructed in the proof of Theorem \ref{thm3-10} 
is described by a transformation of $P$, but we have the following proposition. 

\begin{proposition}
There exist lattice presentations of partial matchings $\Delta, \Delta'$ and a transformation $\Delta \to \Delta'$ by rectangles $R_j \cup (-R_j^*)$ $(j=1,\ldots,k)$ such that 
$\sum_{j=1}^k |\mathrm{Area}(R_j)|=\mathrm{Area}|P|$  
but $R_1=R(v,w)$ satisfies $v \in  \mathrm{Ver}_0(P)$ and $w \in  \mathrm{Ver}_0(-P^*)$. Thus, there exists a transformation $\Delta \to \Delta'$ with minimal area $\mathrm{Area}|P|$ which is not presented by a transformation of $P$. We have examples satisfying one of the following conditions. 

\begin{enumerate}[$(1)$]\label{prop4-13}
\item 
The lattice polytope $P$ is connected and simple. In this case, there is also a transformation $\Delta \to \Delta'$ with minimal area which is presented by a transformation of $P$. 

\item
The lattice polytope $P$ is connected but not simple. In this example, any transformation $\Delta \to \Delta'$ with minimal area is not presented by a transformation of $P$. 

\end{enumerate}
\end{proposition}


\begin{proof}
Case (1). We consider lattice presentations with associated lattice polytopes as illustrated  in Figure \ref{fig4-8a}, where the shadowed polytope is $P$, and the vertices of $\mathrm{Ver}_0(P \cup (-P^*))$ (respectively $\mathrm{Ver}_1(P\cup (-P^*))$) are indicated by black circles (respectively X marks), and the numbers in regions divided by $\partial (P\cup (-P^*))$ denote the rotation numbers. Then, the transformation as shown in Figure \ref{fig4-9a} satisfies $\sum_j|\mathrm{Area}(R_j)|=\mathrm{Area}|P|$ but $R_1=R(v,w)$ satisfies $v \in  \mathrm{Ver}_0(P)$ and $w \in  \mathrm{Ver}_0(-P^*)$. The latter statement is obvious, and also follows from Theorem \ref{thm3-10}. 

Case (2). 
We consider lattice presentations with associated lattice polytopes as illustrated  in Figure \ref{fig4-8a} and the transformation as shown in Figure \ref{fig4-9a} satisfies the required condition. The latter statement is obvious, since the minimal area of the lattice polytope $P$ is the sum of the two areas sorrounded by $\partial P$ by Theorem \ref{thm3-10}. 
\end{proof}

\begin{figure}[ht]\centering
\includegraphics*[height=3cm]{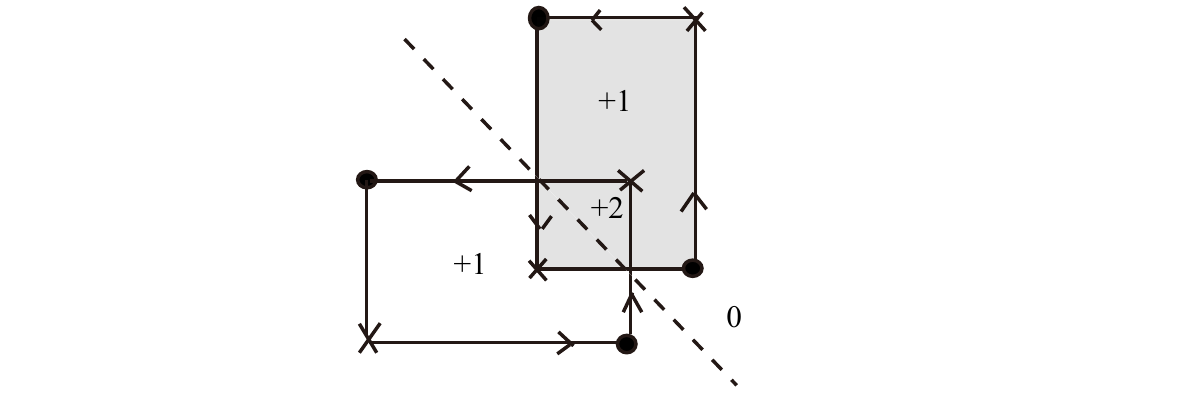}
\caption{Example of lattice polytopes $P \cup (-P^*)$ of Proposition \ref{prop4-13} Case (1), where the vertices of $\mathrm{Ver}_0(P\cup (-P^*))$ (respectively $\mathrm{Ver}_1(P\cup (-P^*))$) are indicated by black circles (respectively X marks), and the numbers in regions divided by $\partial P$ denote the rotation numbers.}
\label{fig4-8a}
\end{figure}

\begin{figure}[ht]
\centering
\includegraphics*[height=2.5cm]{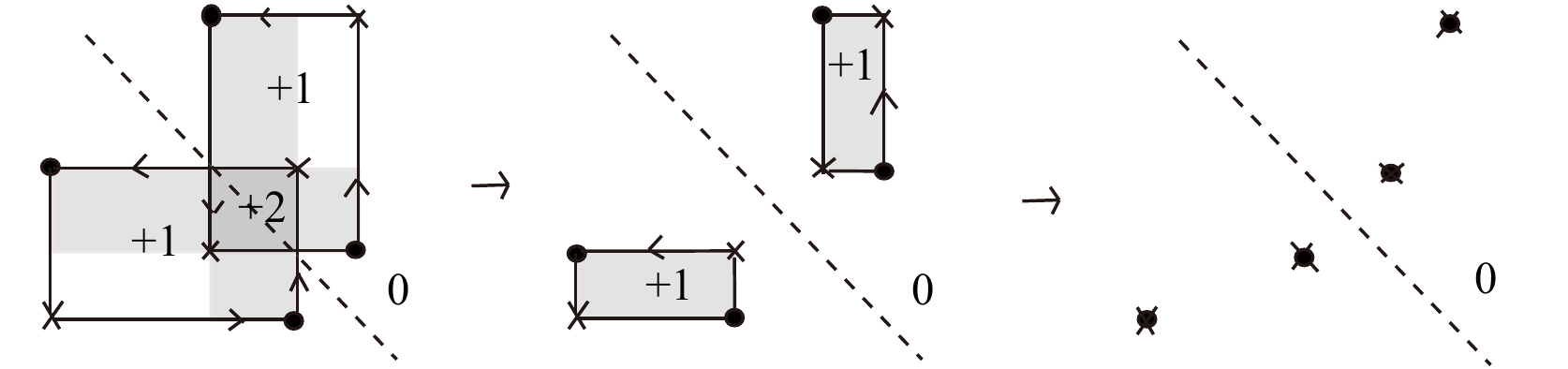}
\caption{Example of a transformation of lattice polytopes $P \cup (-P^*)$ given by Figure \ref{fig4-8a}, which has a minimal area but is not presented by a transformation of $P$, where the vertices of $\mathrm{Ver}_0(P\cup (-P^*))$ (respectively $\mathrm{Ver}_1(P\cup (-P^*))$) are indicated by black circles (respectively X marks).}
\label{fig4-9a}
\end{figure}

\begin{figure}[ht]\centering
\includegraphics*[height=4cm]{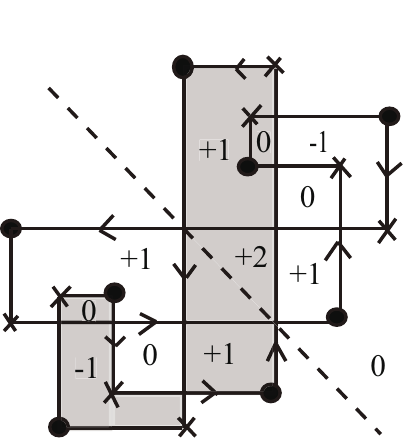}
\caption{Example of lattice polytopes $P \cup (-P^*)$ of Proposition \ref{prop4-13} Case (2), where the vertices of $\mathrm{Ver}_0(P\cup (-P^*))$ (respectively $\mathrm{Ver}_1(P\cup (-P^*))$) are indicated by black circles (respectively X marks), and the numbers in regions divided by $\partial P$ denote the rotation numbers.}
\label{fig4-8}
\end{figure}

\begin{figure}[ht]
\centering
\includegraphics*[height=7cm]{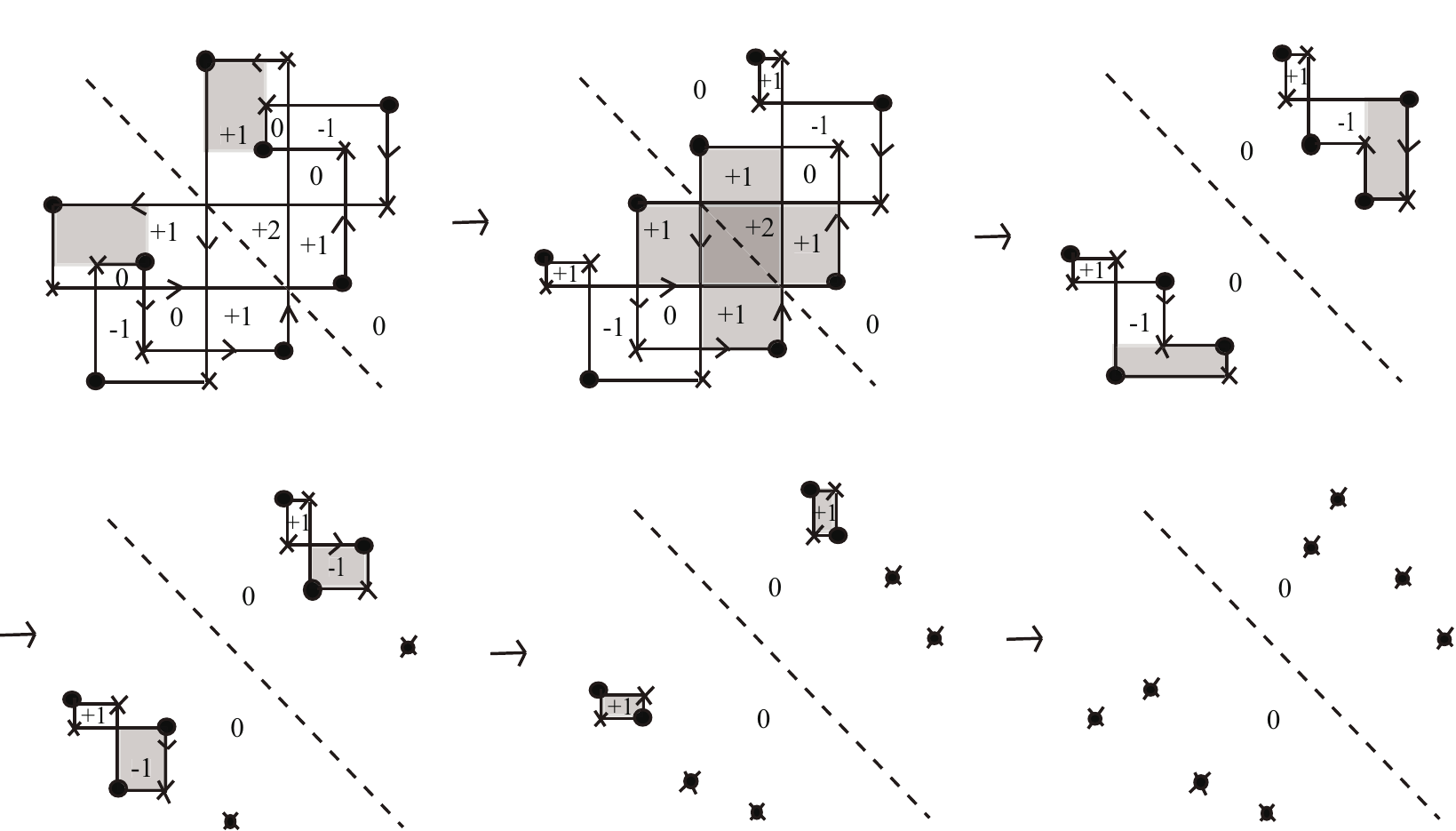}
\caption{Example of a transformation of lattice polytopes $P \cup (-P^*)$ given by Figure \ref{fig4-8}, which has a minimal area but is not presented by a transformation of $P$, where the vertices of $\mathrm{Ver}_0(P\cup (-P^*))$ (respectively $\mathrm{Ver}_1(P\cup (-P^*))$) are indicated by black circles (respectively X marks).}
\label{fig4-9}
\end{figure}

\begin{corollary}\label{cor4-12}
Let $\Delta, \Delta'$ be two lattice presentations of partial matchings. 
We consider a transformation $\Delta \to \Delta'$ with minimal area. Then, if $P$ and $-P^*$ are disjoint and $P$ satisfies the condition $(1)$ or $(2)$ of Theorem \ref{thm3-10}, then $\Delta \to \Delta'$ is described by a transformation of $P$. 
Further, in this case, the transformations consist of those which, as chord diagrams, changes nesting arcs to crossing arcs and vice versa. 
\end{corollary}

\begin{proof}
By Lemma \ref{lem3-14} and the proof of Theorem \ref{thm3-10}, if $P$ and $-P^*$ are disjoint and $P$ satisfies the condition $1$ or $2$ of Theorem \ref{thm3-10}, then $\Delta \to \Delta'$ is described by a transformation of $P$. 

Since $P$ and $-P^*$ are disjoint, any used rectangle $R$ is disjoint with $-R^*$, and it follows from Proposition \ref{prop2-3} that the two arcs of presented chord diagram are non-separated before and after the transformation, and hence we see that the transformation changes nesting arcs  to crossing arcs and vice versa.  
\end{proof}

Let $f(\Delta, \Delta')$ (respectively $f(P(\Delta, \Delta'))$) be the number of transformations $\Delta \to \Delta'$ with minimal area (respectively the number of transformations of $P=P(\Delta, \Delta')$ with minimal area). Note that by definition,  $f(\Delta, \Delta')=f(P\cup (-P^*))$. 

By Lemma \ref{lem3-14} and the existence of a transformation with minimal area by Theorem  \ref{thm3-10}, we have the following corollary. 
\begin{corollary}\label{cor3-15}
 If a lattice polytope $P=P(\Delta, \Delta')$ satisfies the condition $1$ or $2$ of Theorem \ref{thm3-10}, then  
the number of transformation with minimal area $f(\Delta, \Delta')$ is equal to $f(P)$. Further, by definition of equivalence, for equivalent lattice polytopes $P$ and $P'$, $f(P)=f(P')$. 
Thus, for pairs of lattice presentations $\Delta_i$, $\Delta_i'$ $(i=1,2)$ such that their lattice polytopes $P_i=P(\Delta_i, \Delta_i')$ $(i=1,2)$ satisfies that $P_i$ and $-P_i^*$ are disjoint and $P_i$ satisfies the condition $1$ or $2$ of Theorem \ref{thm3-10} $(i=1,2)$, 
$f(\Delta_1, \Delta_1')=f(\Delta_2, \Delta_2')$ if $P_1$ and $P_2$ are equivalent. 
\end{corollary}

\section{Reduced graphs}\label{sec-red}
In this section, we introduce the notion of the reduced graph of a lattice polytope. Then we can determine whether or not a lattice polytope has a transformation with minimal area by studying its reduced graph (Theorem \ref{thm-red}). 

\begin{definition}\label{def-graph}
For a lattice polytope $P$, let $\partial P$ be the boundary of $P$ equipped with the initial vertices $\mathrm{Ver}_0(P)$, which are one of the two types $\mathrm{Ver}_0(P)$ and $\mathrm{Ver}_1(P)$, and with orientations of edges, and labels assigned to regions divided by $\partial P$ denoting the rotation number of each region. We regard $\partial P$ as a graph of a finite numebr of immersed oriented circles with transverse intersection points and equipped with several vertices and integral labels for each divided region. 
Then, we consider the following deformations, where an {\it arc} is  a connected component of $\partial P$ minus the intersection points. 

\begin{enumerate}[(I)]
\item
Reduce several vertices on an arc to one vertex on the arc (see Figure \ref{2017-0516-01} (I)).
\item
The local move illustrated in Figure \ref{2017-0516-01} (II).
\item
The local move illustrated in Figure \ref{2017-0516-01} (III).
\item
Remove a closed arc bounded with a 2-disk $E$ whose interior is disjoint with the graph and the label assigned to $E$ is not zero (see Figure \ref{2017-0516-01} (IV)).
\begin{figure}[ht]
\centering
\includegraphics*[height=4cm]{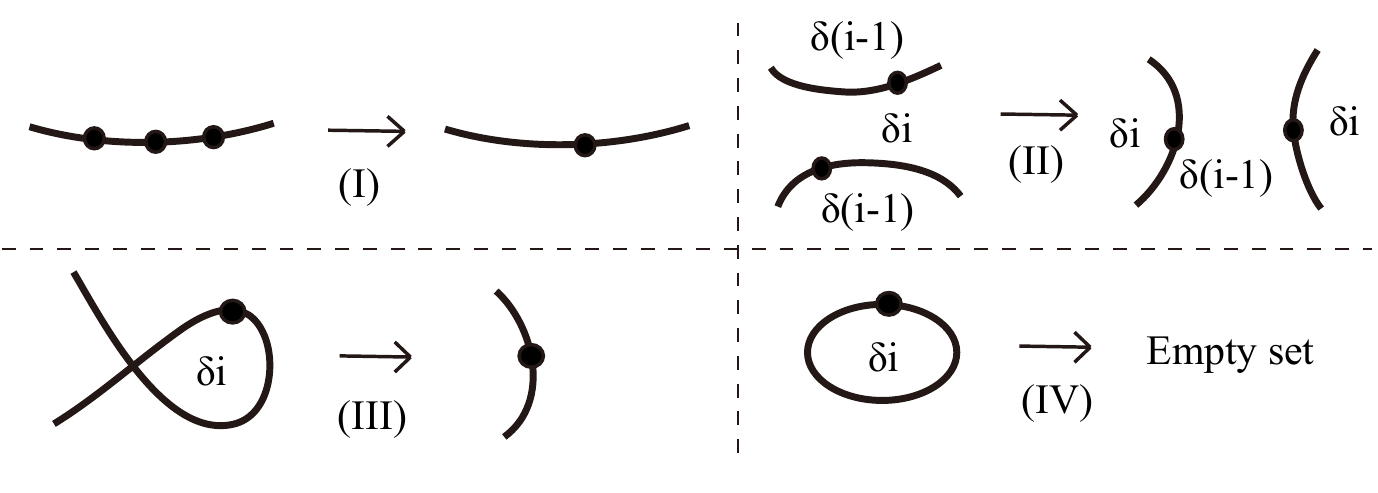}
\caption{The local deformations (I)--(IV), where $\delta \in \{+1, -1\}$ and $i$ is a positive integer, and we omit the orientationsof the arcs and some of the labels of the regions. The deformation (II) is applicable only if the resulting graph admits the induced orientation.}
\label{2017-0516-01}
\end{figure}
\end{enumerate}

We consider a graph obtained from the deformations (I)--(IV) such that no more deformations can be applied. It is unique up to an ambient isotopy of $\mathbb{R}^2$ by Lemma \ref{lem0515}. We call the graph the {\it reduced graph} of a lattice polytope $P$. 
\end{definition}

For example, we obtain in Figure \ref{2017-0517-01} the reduced graph of the lattice polytope given in the leftmost figure of Figure \ref{fig4-7} in Proposition \ref{prop3-18}. 
\begin{figure}[ht]
\centering
\includegraphics*[width=13cm]{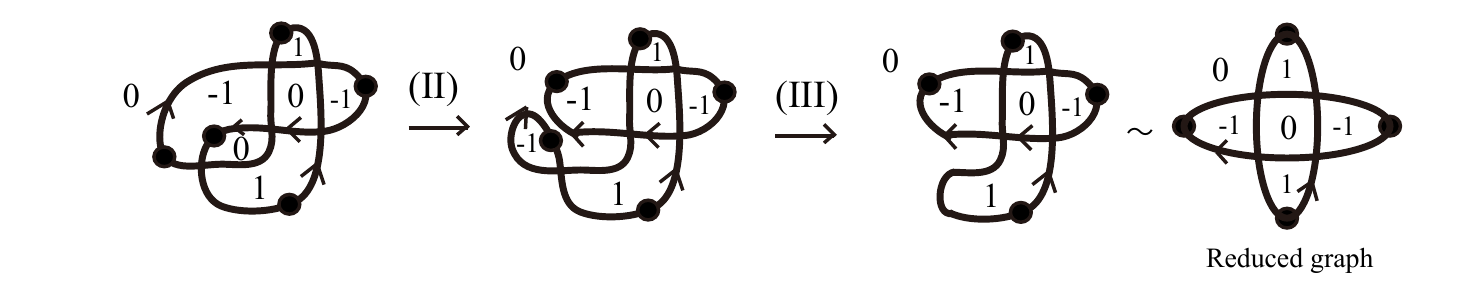}
\caption{Obtaining the reduced graph of the lattice polytope given in the leftmost figure of Figure \ref{fig4-7}.}
\label{2017-0517-01}
\end{figure}

\begin{lemma}\label{lem0515}
The reduced graph of a lattice polytope is unique up to an ambient isotopy of $\mathbb{R}^2$.
\end{lemma}

\begin{proof}
It is obvious that the deformation (I) commutes the deformation (II). Consider a local graph as illustrated in the left figure of Figure \ref{2017-0517-02}. Let $\delta i$ be the label of the region encircled by an arc whose closure is  a circle, where $\delta \in \{+1, -1\}$ and $i$ is a positive integer. Then, the label of the adjoining region is $\delta (i-1)$. 
Hence there does not exists a local graph where both deformations (II) and (III) are applicable; see  Figure \ref{2017-0517-02}. 
Consider a local graph where both deformations (II) and (IV) are applicable. Then, the result by the deformation (II) is equal to the result by deformation (IV) after applying the deformation (I) as illustrated in Figure \ref{2017-0517-03}. 
Thus the reduced graph is unique. 
\end{proof}

\begin{figure}[ht]
\centering
\includegraphics*[height=2cm]{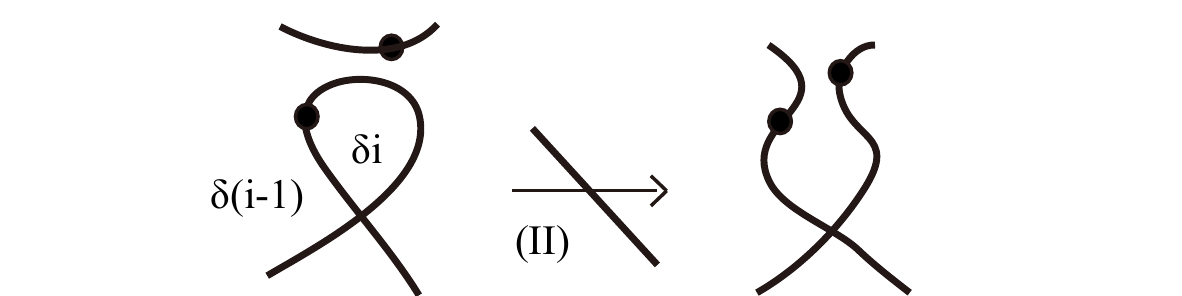}
\caption{The deformation (II) cannot be applied.}
\label{2017-0517-02}
\end{figure}
\begin{figure}[ht]
\centering
\includegraphics*[height=3cm]{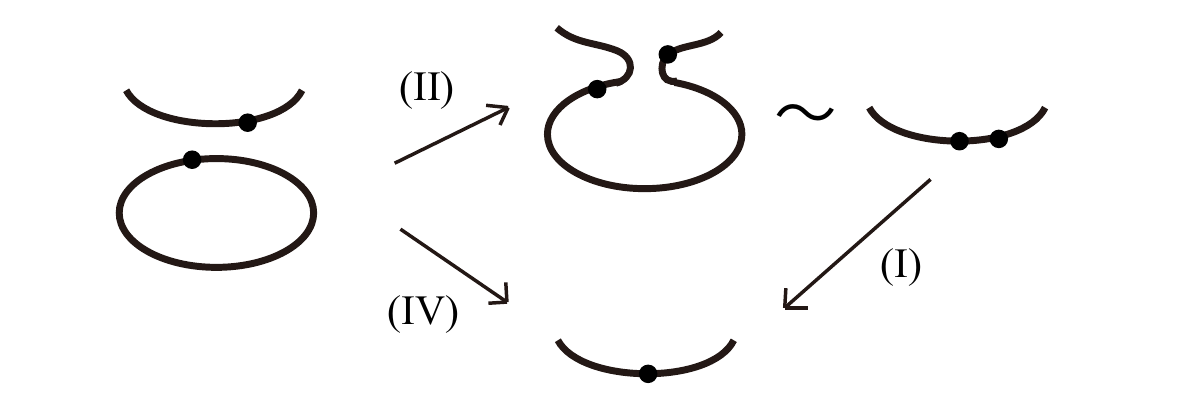}
\caption{The relation between deformations (II) and (IV).}
\label{2017-0517-03}
\end{figure}

By the proof of Theorem \ref{thm3-10}, we have the following theorem. We give the proof at the end of the next section. 
\begin{theorem}\label{thm-red}
Let $P$ be a lattice polytope. 
Then, the reduced graph of $P$ is an empty graph if and only if there exists a transformation of $P$ with the minimal area $\mathrm{Area}|P|$. 
\end{theorem}

\section{Lemmas and Propositions}\label{sec4-3}
For a point $u$ of $\mathbb{R}^2$, we denote by $x(u)$ and $y(u)$ the $x$ and $y$-components of $u$, respectively. 

\begin{lemma}\label{lem4-15}
For a transformation of a lattice polytope $P$ by rectangles $R_j$, 
the union of rectangles $R_j$ form regions whose boundaries are the boundary of $P$. 
\end{lemma}

\begin{proof}
It suffices to show the case $P$ is connected. 
Put $P_0=R_1$ and let $P_j$ be a lattice polytope with region $P_j=P_{j-1} \cup R_j$
and $\mathrm{Ver}_1(P_j)=t(\mathrm{Ver}_1(P_{j-1})\cup \{v,w\}, R(v,w))$. For $P_0=R_1$, $R_1$ forms a region of $P_0$. Hence, by induction, it suffices to show that $P' \cup R(v,w)$ forms a region of a lattice polytope $P$, for a lattice polytope $P'$ and a rectangle $R(v,w)$ such that $\mathrm{Ver}_1(P)=t(\mathrm{Ver}_1(P')\cup \{v,w\}, R(v,w))$. Since we assume that $P$ is connected, at least one of $v,w$ is in $\mathrm{Ver}_1(P')$. We have two cases: (Case 1) $v \in \mathrm{Ver}_1(P')$ and $w \not\in \mathrm{Ver}_1(P')$, and (Case 2) $v,w \in \mathrm{Ver}_1(P')$. 
In both cases, $I=\partial P' \cap \partial R(v,w)$ is either $\{v\}$, or $\{v,w\}$ or an interval or intervals of $\partial R(v,w)$ containing $v$ or $w$. The orientation of edges of $\partial P'$ in the $x$-direction (respectively $y$-direction) is toward $v$ or $w$ (respectively from $v$ or $w$), and the orientation of edges of $\partial R(v,w)$ in the $x$-direction (respectively $y$-direction) is from $v$ or $w$ (respectively toward $v$ or $w$). Hence the orientations of the intervals of $I \subset \partial P'$ and $I \subset \partial R(v,w)$ are opposite, and the intervals are canceled when we take a union of rectangles, to form one region. Further, put $\tilde{v}=(x(w), y(v))$ and $\tilde{w}=(x(v), y(w))$, the other diagonal vertices of $R(v,w)$. Then, the region $P' \cup R(v,w)$ forms a lattice polytope $P$ with $\mathrm{Ver}_0(P)=\mathrm{Ver}_0(P') \cup \{w\}$ and $\mathrm{Ver}_1(P)=(\mathrm{Ver}_1(P')\backslash \{v\}) \cup \{\tilde{v}, \tilde{w}\}=t(\mathrm{Ver}_1(P')\cup \{v,w\}, R(v,w))$ for Case (1) (see Figure \ref{fig4-10}), and $\mathrm{Ver}_0(P)=\mathrm{Ver}_0(P')$ and $\mathrm{Ver}_1(P)=(\mathrm{Ver}_1(P')\backslash \{v,w\}) \cup \{\tilde{v}, \tilde{w}\}=t(\mathrm{Ver}_1(P')\cup \{v,w\}, R(v,w))$ for Case (2) (see Figure \ref{fig4-11}). Thus we have the required result. 
\end{proof}

\begin{figure}[ht]
\centering
\includegraphics*[height=8cm]{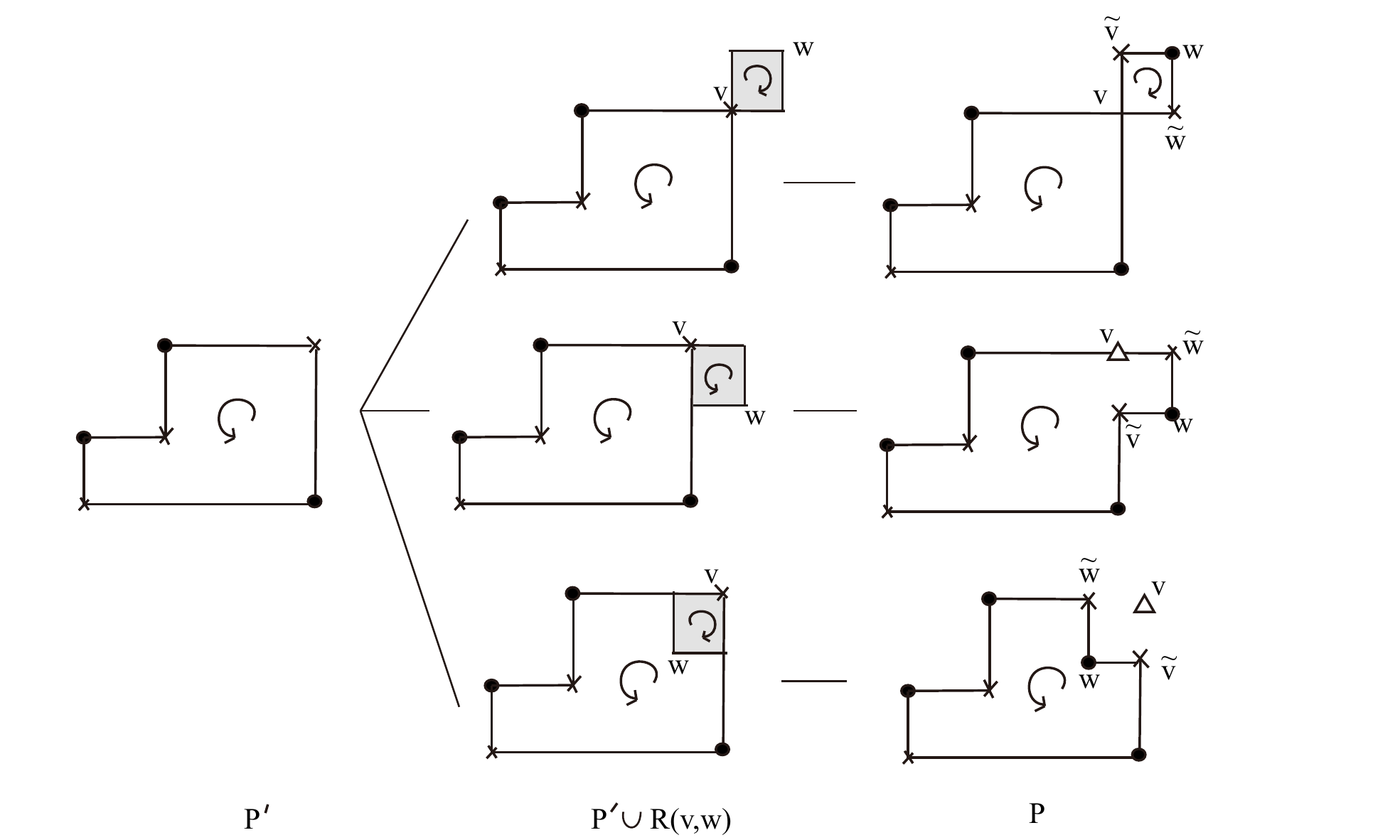}
\caption{Examples of Case (1) of lattice polytopes $P'$ and $P$ with region $P=P' \cup R(v,w)$ and $\mathrm{Ver}_1(P)=t(\mathrm{Ver}_1(P')\cup \{v,w\}, R(v,w))$, with $v \in \mathrm{Ver}_1(P')$ and $w \not\in \mathrm{Ver}_1(P')$, 
where the vertices of $\mathrm{Ver}_0(P')$, $\mathrm{Ver}_0(P)$ (respectively $\mathrm{Ver}_1(P')$, $\mathrm{Ver}_1(P)$) are indicated by black circles (respectively X marks). }
\label{fig4-10}
\end{figure}

\begin{figure}[ht]
\centering
\includegraphics*[height=6cm]{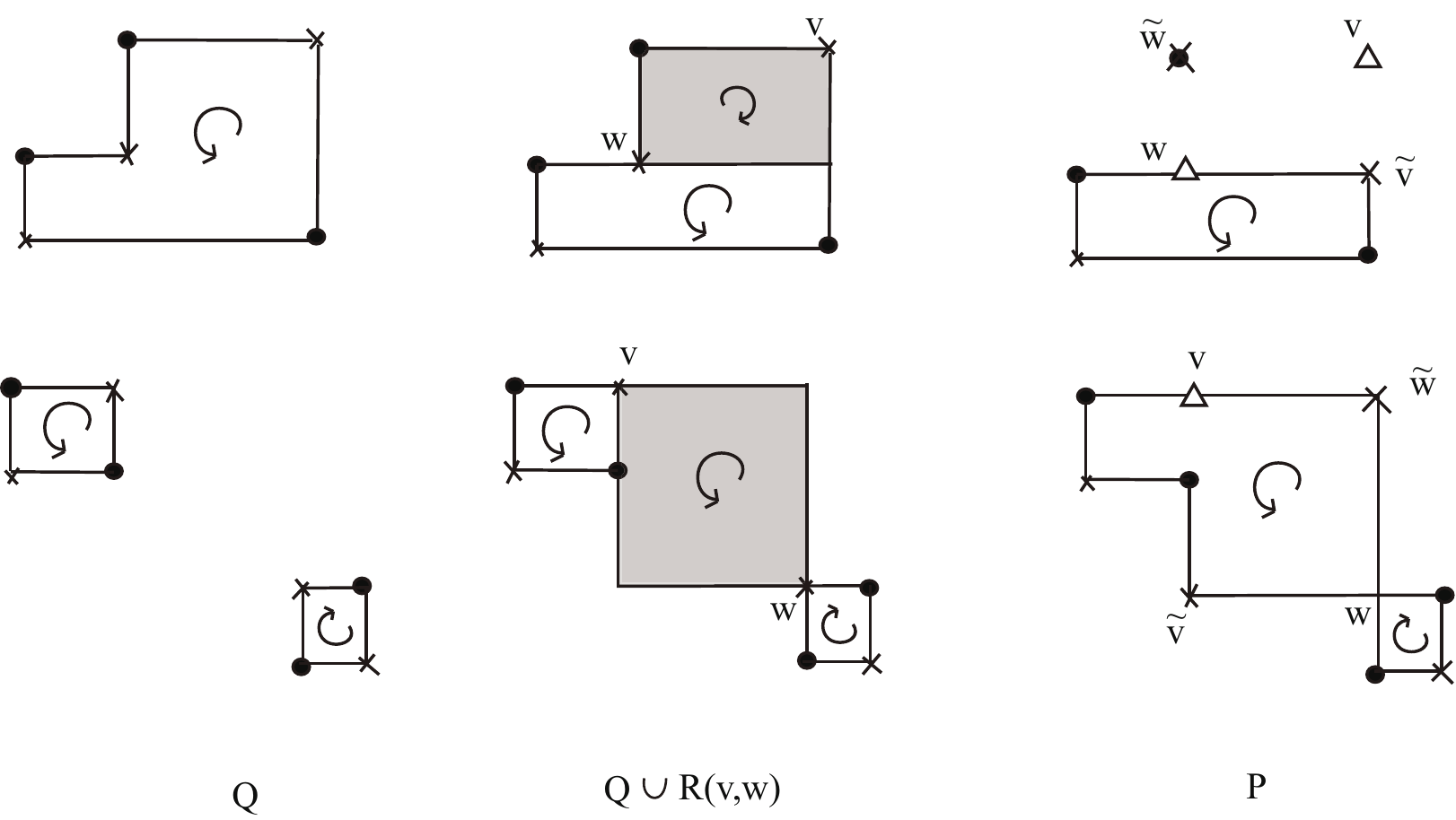}
\caption{Examples of Case (2) of lattice polytopes $P'$ and $P$ with region $P=P' \cup R(v,w)$ and $\mathrm{Ver}_1(P)=t(\mathrm{Ver}_1(P')\cup \{v,w\}, R(v,w))$, with $v,w \in \mathrm{Ver}_1(P')$, 
where the vertices of $\mathrm{Ver}_0(P')$, $\mathrm{Ver}_0(P)$ (respectively $\mathrm{Ver}_1(P')$, $\mathrm{Ver}_1(P)$) are indicated by black circles (respectively X marks). }
\label{fig4-11}
\end{figure}

\begin{proposition}\label{lem3-15}
Let $P$ be a simple connected lattice polytope. 
Then, there exists a rectangle $R(v,w)$ $(v,w \in \mathrm{Ver}_0(P))$ contained in $P$ such that $t(P, R(v,w))$ is a simple lattice polytope with region $P\backslash R(v,w)$. 
\end{proposition}

\begin{proof}
By Lemma \ref{lem3-18} and \ref{lem3-19}, we have the required result. 
\end{proof}

\begin{lemma}\label{lem3-18}
Let $P$ be a simple connected lattice polytope.  
Then, there exists a rectangle $R(v,w)$ $(v,w \in \mathrm{Ver}_0(P))$ contained in $P$. \end{lemma}

For a vertex $v \in \mathrm{Ver}_0(P)$, we denote by $v'$ the vertex of $\mathrm{Ver}_1(P)$ such that the edge $\overline{vv'}$ is in the $x$-direction. Recall that 
$\mathbf{e}_1=(1,0)$ and $\mathbf{e}_2=(0,1)$, 
the standard basis of the $xy$-plane $\mathbb{R}^2$. 

\begin{proof}
For $v \in \mathrm{Ver}_0(P)$, take a point $u \in \partial P$ such that $u=v+x \mathbf{e}_1$ for some $x$ and $\overline{v u} \cap \overline{vv'} \neq \{v\}$ and $\overline{v u} \subset P$. This point $u$ is unique. Then, take  
another point $u'=u+y\mathbf{e}_2$ for some $y$ such that $R(v, u') \subset P$ and $(u'+\mathbb{R}\mathbf{e}_1) \cap \partial P$ consists of edges of $\partial P$. Note that there may be several choice of $u'$. Fix $u'$. Then, 
since an edge of $P$ in the $x$-direction with a fixed $x$-component is unique, $(u'+\mathbb{R}\mathbf{e}_1) \cap \partial P=\overline{ww'}$ for some unique $w\in \mathrm{Ver}_0(P)$ and $w' \in \mathrm{Ver}_1(P)$. 
Note that when $u'$ is a vertex of $P$, then $u' \in \mathrm{Ver}_0(P)$ and $w=u'$. 
By construction, $R(v,w)$ is the required rectangle. 
For $v$, we can construct another rectangle $R(v,z)$ in a similar way by first taking an interval in the $y$-direction and then making it fat in the $x$-direction. See Figure \ref{fig4-12}. 
\end{proof}

\begin{figure}[ht]
\centering
\includegraphics*[height=4cm]{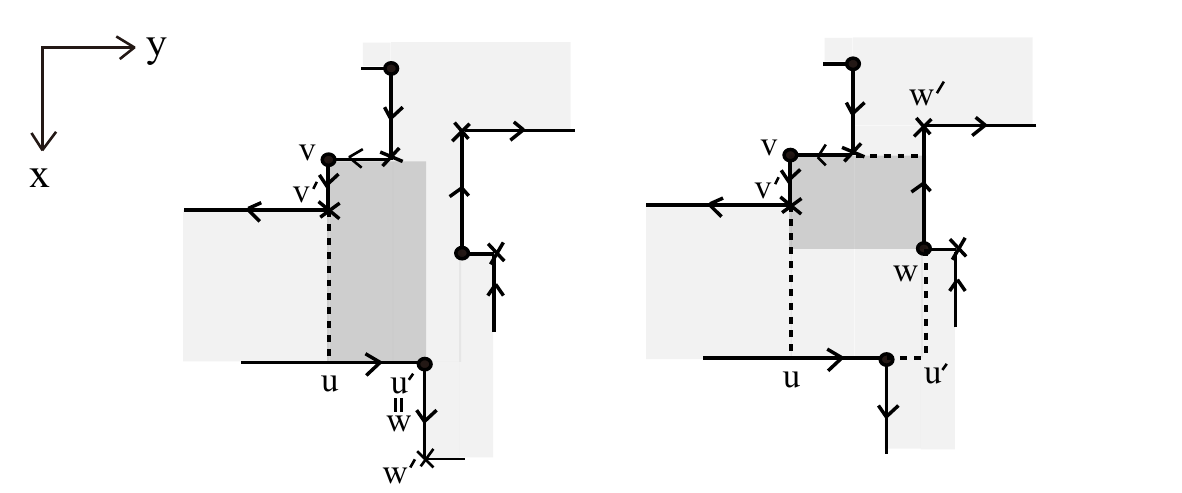}
\caption{The vertices $v, v', u, u', w$ and $w'$, where the vertices of $\mathrm{Ver}_0(P)$ (respectively $\mathrm{Ver}_1(P)$) are indicated by black circles (respectively X marks).}
\label{fig4-12}
\end{figure}

\begin{lemma}\label{lem3-19}
Let $P$ be a connected simple lattice polytope and let $R(v,w)$ be the rectangle constructed in Lemma \ref{lem3-18}. 
Then, the result of transformation $t(P, R(v, w))$ is a simple lattice polytope with region $P\backslash R(v,w)$.  
\end{lemma}

\begin{proof}
Put $R=R(v,w)$. 
Let $u$ be a point in $\partial P$ as in the proof of Lemma \ref{lem3-18}. 
First we see that $\partial R \cap \partial P$ consists of $\partial R=\partial P$, one interval, or two intervals. Since $v,w$ are vertices of $P$ and an edge of $P$ in the $x$-direction (respectively $y$-direction) with a fixed $x$-component (respectively $y$-component) is unique, $\partial R \cap \partial P$ consists of the union of points $v$, $w$, and intervals containing $v$ or $w$. Hence it suffices to show that there are intervals in $\partial R \cap \partial P$ containing $v$ and $w$. 
Assume $x(v)<x(v')$.  Put the other pair of diagonal points of $R(v,w)$ by $\tilde{v}=(x(w), y(v))$ and $\tilde{w}=(x(v), y(w))$.  

Recall that we give an orientation of $P$ by giving each edge in the $x$-direction (respectively in the $y$-direction) the orientation from a vertex of $\mathrm{Ver}_0(P)$ to a vertex of $\mathrm{Ver}_1(P)$ (respectively from a vertex of $\mathrm{Ver}_1(P)$ to a vertex of $\mathrm{Ver}_0(P)$), and $\partial P$ has a coherent orientation as an immersion of a circle. 
Now, consider a point moving continuously in an interval. 
When a point $p$ passes a point $q$ from one direction and passes $q$ again for the second time, $p$ comes back from the other direction. Since $P$ is simple, this implies with the assumption $x(v)<x(v')$, that the situation in which both $x(w)<x(w')$ and $x(w)<x(u)$ do not occur simultaneously; thus, if $x(w)<x(w')$, then $x(w)=x(u)=x(u')$, hence $w=u'$. Thus, by construction, we see that if $x(w)<x(w')$, then $\tilde{v}=u$. 
 
When $x(w)<x(w')$, then $\tilde{v}=u$ and $w=u'$, hence $\overline{uu'}=\overline{\tilde{v} w}$ is an interval in $\partial R \cap \partial P$ containing $w$. 
When $x(w)>x(w')$, by construction, $x(w)\geq x(\tilde{w})$, hence $\overline{w\tilde{w}}\cap \overline{ww'}$ is an interval in $\partial R \cap \partial P$ containing $w$. Further, in both cases,  $\overline{vv'}$ is an interval in $\partial R \cap \partial P$ containing $v$. 
Thus $\partial R \cap \partial P$ consists of $\partial R=\partial P$, one interval, or two intervals. 
  
Since $\partial P$ has an orientation, we denote the vertices of $\mathrm{Ver}_0(P)$ and $\mathrm{Ver}_1(P)$ by $v_1,\ldots, v_n$ and $v'_1, \ldots, v_n'$ such that $\overline{v_j v_j'}$ $(j=1,\ldots, n)$ is an edge in the $x$-direction and the vertices  appear by $v_1, v_1', v_2, v_2' \ldots, v_n, v_n'$ on $\partial P$ with respect to the orientation.  
In the above argument, we assume that $v=v_i$ and $w=v_j$ with $i<j$.

Then, let $P_1$ and $P_2$ be lattice polytopes determined by vertices 
\[
v_1, v_1', \ldots, v_{i-1}, v'_{i-1}, \tilde{v}_j, v_j', v_{j+1}, \ldots, v_n, v_n'\]
 and 
 \[
 \tilde{v}_i, v_i', v_{i+1}, v_{i+1}', \ldots, v_{j-1}, v_{j-1}',\]
  respectively such that 
\begin{eqnarray*}
&&\mathrm{Ver}_0(P_1)=\{ v_1, \ldots, v_{i-1}, \tilde{v}_j, v_{j+1},   \ldots, v_n \}, \\
&& \mathrm{Ver}_1(P_1)=\{v_1', v'_{i-1},  v_j',  \ldots, v_n'\}, \\
&&\mathrm{Ver}_0(P_2)=\{ \tilde{v}_i, v_{i+1}, \ldots, v_{j-1} \},\\
 && \mathrm{Ver}_1(P_2)=\{v_i', v_{i+1}', \ldots, v_{j-1}'\},  
 \end{eqnarray*}
see Figure \ref{fig4-13}. 
 Since both $P$ and $R=R(v_i,v_j)$ are simple and connected, $P_1$ and $P_2$ are simple connected lattice polytopes. 
Since $\partial R \cap \partial P$ consists of either $\partial R=\partial P$, one interval, 
or two intervals containing $v$ and $w$, 
$P_1 \cap P_2=\emptyset$. Note that each of $P_1$ and $P_2$ consists of one isolated vertex (respectively one of $P_1$ and $P_2$ consists of one isolated vertex) when $\partial R=\partial P$ (respectively $\partial R \cap \partial P$ consists of one interval). 
Hence the region of $P_1 \cup P_2$ is formed by $P\backslash R(v,w)$, and we have the required result. 
\end{proof}

\begin{figure}[ht]\centering
\includegraphics*[height=4cm]{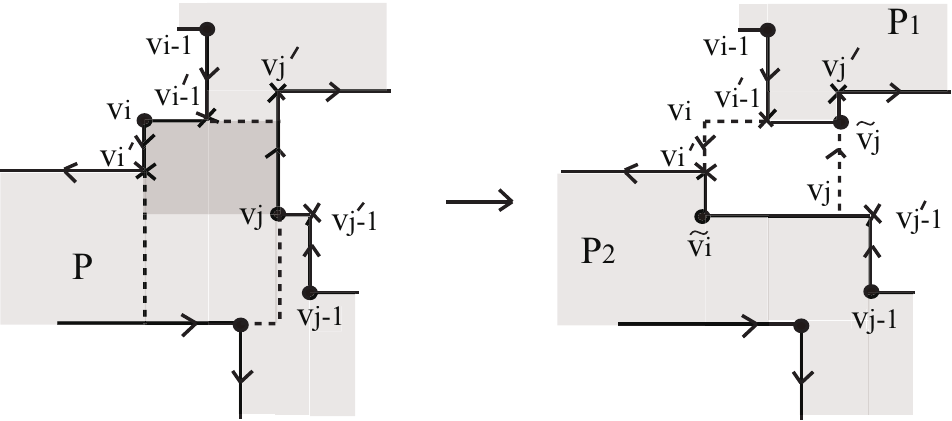}
\caption{Lattice polytopes $P$, $P_1$ and $P_2$, where the vertices of $\mathrm{Ver}_0(Q)$ (respectively $\mathrm{Ver}_1(Q)$), $Q=P, P_1, P_2$, are indicated by black circles (respectively X marks).}
\label{fig4-13}
\end{figure}

Let $\Delta=\Delta_0 \to \Delta_1 \to \cdots \to \Delta_k=\Delta'$ be a transformation with $\Delta_j=t(\Delta_{j-1}, R_j\cup (-R_j^*))$ for a rectangle $R_j=R_j(v_j, w_j)$ $(j=1,2,\ldots,k)$. 
For a lattice polytope $P$ associated with $\Delta$ and $\Delta'$, put $V_0=\mathrm{Ver}_0(P)$. 
We define $V_j$ inductively by $V_{j}=t(V_{j-1}, R_j(v_j, w_j))$ $(j=1,2,\ldots,k)$, 
if the diagonal vertices $v_{j}, w_{j}$ of $R_{j}(v_{j}, w_{j})$ satisfy $v_{j}, w_{j} \in V_{j-1}$. Note that $\Delta_j=V_j \cup (-V_j^*)$ if $V_j$ can be defined, and $V_k=\mathrm{Ver}_1(P)$. 
Then we have the following. 

\begin{lemma}\label{lem3-14}
If  $P$ and $-P^*$ are disjoint and $P$ satisfies the condition $1$ or $2$ of Theorem \ref{thm3-10} and further the area of the transformation $\Delta \to \Delta'$ is minimal, then $V_j$ can be defined for all $j=1,\ldots,k$.  
\end{lemma}

\begin{proof}
Assume that the area is minimal. 
By the proof of Theorem \ref{thm3-10}, the area is minimal when the used rectangles can be divided to several sets such that each set of rectangles form each simple lattice polytope or a set of simple lattice polytopes which bound a region obtained from a 2-disk $D^2$ by removing several mutually disjoint disks in the interior of $D^2$. Thus the vertices of each rectangle are contained in one of such regions. Since $P$ and $-P^*$ are disjoint, each of these regions is contained in either the region of $P$ or the region of $-P^*$. If $v_1 \in V_{0}=\mathrm{Ver}_0(P)$ and $w_1 \in -V_{0}^*$, then $R(v_1, w_1)$, together with other rectangles, forms a region which contains a vertex $v_1$ of $P$ and a vertex $w_1$ of $-P^*$. Since the components of $P$ and $-P^*$ are distinct, this is a contradiction, and hence we can assume that $v_1, w_1 \in V_0$ and we have $V_1$. Since the area of $\Delta_0 \to \Delta_1 \to \cdots \to  \Delta_k$ is minimal, the area of the transformation $\Delta _1=V_1 \cup (-V_1^*) \to \cdots \to \Delta_k$ is also minimal. Hence, by repeating the same argument, we see that $V_j$ can be defined for all $j=1,\ldots,k$.  
\end{proof}

\begin{proof}[Proof of Theorem \ref{thm-red}]
We consider local deformations (I)--(IV) given in Definition \ref{def-graph}, but instead of deformations (III) and (IV), which contain one vertex, we consider (III) as the local move with one or two vertices in the arc whose closure is a circle, and (IV) as the local move with two vertices on the closed arc. Then, by the proof of Theorem \ref{thm3-10}, we see that all possible transformations which realize the minimal area are in the corresponding graphs presented by deformations (I)--(IV). 

From a transformation of a lattice polytope by rectangles with minimal area, we obtain a sequence of graphs related by deformations (I)--(IV). Thus, the reduced graph of a lattice polytope $P$ is an empty graph if there exists a transformation of $P$ with the minimal area. Conversely, since all possible transformations which realize the minimal area are presented by deformations (I)--(IV), if there does not exist a transformation of $P$ with the minimal area, then the reduced graph is not empty. Thus we have the required result.
\end{proof}

\section{Simple lattice polygons}\label{sec5}
In this section, we consider simple lattice polytopes with one component, which we call simple {\it lattice polygons}. 
 By the proof of Proposition \ref{lem3-15}, we have the following. 
 
\begin{proposition}\label{prop3-20}
Let $P$ be a simple lattice polygon and let $R(v,w)$ $(v,w \in \mathrm{Ver}_0(P))$ be a rectangle contained in $P$. 
Then, the result of transformation $t(P, R(v, w))$ is a simple lattice polygon with region $P\backslash R(v,w)$ if and only if $\partial R(v,w)$ contains an interval of $\partial P$.  
\end{proposition}

\begin{proof}
If $\partial R(v, w)$ $(v,w \in \mathrm{Ver}_0(P))$ contains an interval of $\partial P$, then 
$R(v,w)$ is a rectangle constructed by the way shown in Lemma \ref{lem3-18}, and it follows from Lemma \ref{lem3-19} that $t(P, R(v, w))$ is a simple lattice polygon with region $P\backslash R(v,w)$. 

If $\partial R(v,w)$ does not contain an interval of $\partial P$, then $R(v,w) \cap \partial P=\{v,w\}$.  Then, the result $t(P, R(v, w))$, which consists of  
 the lattice polytopes $P_1$ and $P_2$ in the proof of \ref{lem3-19}, satisfies 
 $P_1 \cap P_2=R(v,w)$, hence the region of $t(P, R(v, w))$ is not the region of $P\backslash R(v,w)$. Thus we have the required result. 
 \end{proof}
 
 By Corollary \ref{cor4-12}, 
in particular, we have the following. 
\begin{corollary}\label{cor5-2}
Let $\Delta, \Delta'$ be two lattice presentations of partial matchings such that a lattice polytope $P$ associated with $\Delta, \Delta'$ is a simple lattice polygon. Let $n$ be half the number of the non-isolated vertices of $P$. Then, a division of $P$ into $n$ rectangles $R_1,\ldots, R_n$ presents a transformation $\Delta \to \Delta'$ with minimal area, where $R_1, \ldots, R_n$ satisfy the condition that they induce a transformation of $P$. 

In particular, if $P$ and $-P^*$ are disjoint, then any transformation $\Delta \to \Delta'$ with minimal area is presented by such a division of $P$; further, the transformations consist of those which, as chord diagrams, changes nesting arcs to crossing arcs and vice versa. 
\end{corollary}

\begin{definition}
In the situation of Corollary \ref{cor5-2}, 
we describe a transformation by a division of $P$ into $n$ rectangles $R_1,\ldots, R_n$, where $R_1, \ldots, R_n$ satisfy the condition that they induce a transformation of $P$, 
and assigning each rectangle $R_j$ with the label $j$ ($j=1,\ldots,n$). 
We call such a division of $P$  
the {\it division of a simple lattice polygon $P$} presenting a transformation; see Figure \ref{fig5-1}. 
\end{definition}

\begin{figure} 
\centering
\includegraphics*[height=5cm]{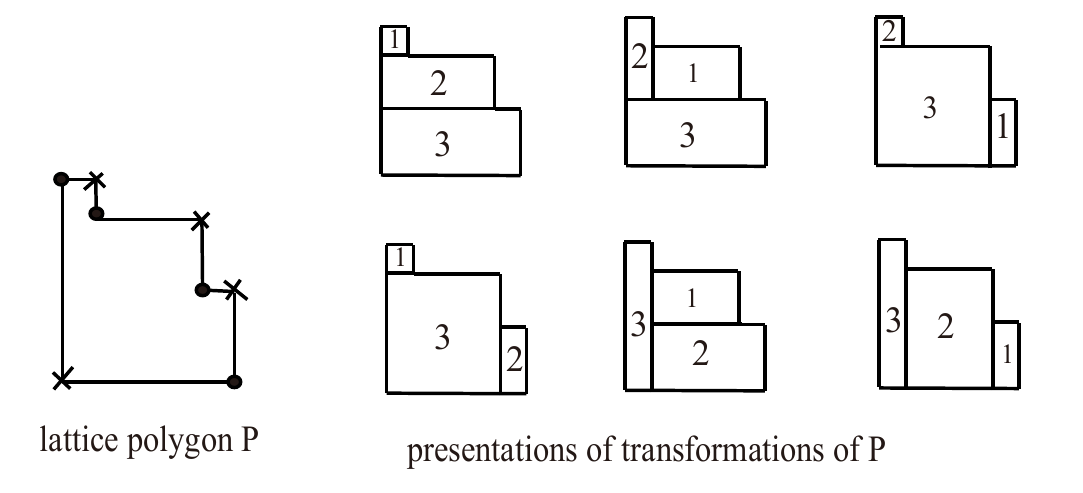}
\caption{Division of a simple lattice polygon $P$ presenting a transformation.}
\label{fig5-1}
\end{figure}
  
\begin{proof}[Proof of Corollary \ref{cor5-2}]
When $n=1$, then $P=R_1$, which is a division by one rectangle. Assume that a simple lattice polygon with $2(n-1)$ vertices is divided by $n-1$ rectangles. By the proof of Lemma \ref{lem3-19}, $P$ is divided to $R \cup P_1 \cup P_2$ for a rectangular $R$ and simple lattice polygons $P_1$ and $P_2$ with $P_1 \cap P_2=\emptyset$ such that the sum of the numbers of vertices of $P_1$ and $P_2$ is $2n$. Then, the number of vertices of $P_1$ and that of $P_2$ is equal to or less than $2(n-1)$, and hence, by assumption, $P$ is divided to $n$ rectangles. Thus, by induction on $n$, together with Corollary \ref{cor4-12}, we have the required result. 
\end{proof}
 
\section*{Acknowledgements}
The author would like to thank Professors Sigeo Ihara and Hiroki Kodama for their helpful comments. 
The author was supported by iBMath through the fund for Platform Project for Supporting in Drug Discovery and Life Science Research (Platform for Dynamic Approaches to Living System) from Japan Agency for Medical Research and Development (AMED), and JSPS KAKENHI Grant Numbers 15K17532 and 15H05740.

\end{document}